\RequirePackage{fix-cm}
\documentclass[smallextended]{svjour3}       
\smartqed  
\usepackage{graphicx}
\usepackage{url}            
\usepackage{booktabs}       
\usepackage{amsfonts}       
\usepackage{nicefrac}       
\usepackage{microtype}      
\usepackage{mathtools}
\usepackage{amsfonts}
\usepackage{mathrsfs}
\usepackage{algorithm}
\usepackage{algorithmic}
\usepackage{color}
\usepackage{graphicx}
\graphicspath{{Figs/}}
\usepackage{indentfirst,latexsym,bm}
\usepackage{longtable}
\usepackage{cuted}
\usepackage{booktabs,multirow,array,multicol}
\usepackage{amsmath}
\usepackage{enumitem}
\newcommand{\STAB}[1]{\begin{tabular}{@{}c@{}}#1\end{tabular}}

\DeclareMathOperator*{\argmin}{argmin}

\usepackage{subfigure}
\usepackage{amssymb}

\usepackage{color}
\definecolor{bluep}{RGB}{0,128,255}
\usepackage{soul,xcolor}
\setstcolor{bluep}
\definecolor{amg}{RGB}{17,140,17}

\newcommand{\Rn}{\mathbb{R}^n}
\newcommand{\R}{\mathbb{R}}
\newcommand{\N}{\mathbb{N}}
\newcommand{\fix}{\mathrm{Fix}}
\newcommand{\prox}{\mathrm{prox}}
\newcommand{\res}[1]{{r_{#1}}}
\newcommand{\pred}[1]{{\mathrm{pred}_{#1}}}
\newcommand{\ared}[1]{{\mathrm{ared}_{#1}}}
\newcommand{\fidweight}{\nu}
\renewcommand{\AA}{\mathsf{AA}}

\newcommand{\citemain}{\cite}
\newcommand{\citeappx}{\cite}
\begin{document}

\title{Nonmonotone Globalization for Anderson Acceleration via Adaptive Regularization
}
\subtitle{}

\titlerunning{Nonmonotone Globalization for Anderson Acceleration}        

\author{Wenqing~Ouyang         \and
        Jiong~Tao	\and
        Andre~Milzarek \and
        Bailin~Deng
}

\authorrunning{Wenqing Ouyang et al.} 

\institute{W. Ouyang \at          
          School of Data Science, The Chinese University of Hong Kong, Shenzhen, China \\
          Shenzhen Research Institute of Big Data, Shenzhen, China\\
          Shenzhen Institute of Artificial Intelligence and Robotics for Society, Shenzhen, China         
           \and
           J. Tao \at
           Department of Computer Science, University of Bath, Bath, UK
           \and
           A. Milzarek \at
           School of Data Science, The Chinese University of Hong Kong, Shenzhen, China\\
           Shenzhen Research Institute of Big Data, Shenzhen, China
           \and
           B. Deng \at
           School of Computer Science and Informatics, Cardiff University, Cardiff, UK
           \and
           Corresponding author: Bailin Deng (\url{DengB3@cardiff.ac.uk})
}

\date{Received: date / Accepted: date}

\maketitle

\begin{abstract}
  Anderson acceleration ($\AA$) is a popular method for accelerating fixed-point iterations, but may suffer from instability and stagnation.
  We propose a globalization method for $\AA$ to improve stability and achieve unified global and local convergence.
  Unlike existing $\AA$ globalization approaches that rely on safeguarding operations and might hinder fast local convergence, we adopt a nonmonotone trust-region framework and introduce an adaptive quadratic regularization together with a tailored acceptance mechanism. 
  We prove global convergence and show that our algorithm attains the same local convergence as $\AA$ under appropriate assumptions. The effectiveness of our method is demonstrated in several numerical experiments.
\keywords{Anderson acceleration \and Global convergence \and Nonmonotone trust region \and Adaptive regularization}
\subclass{65B05 \and 65K10}
\end{abstract}

\section{Introduction}
In  applied mathematics, many problems can be reduced to solving a nonlinear fixed-point equation 
%
$g(x)=x$,
%
where $x\in\mathbb{R}^n$ and $g:\mathbb{R}^n\rightarrow\mathbb{R}^n$ is a given function. If $g$ is a contractive mapping, i.e.,
\begin{equation} \label{eq1-con} 
\|g(x)-g(y)\|\leq \kappa\|x-y\| \quad \forall~x,y\in\mathbb{R}^n,
\end{equation} 
where $\kappa<1$, then the iteration
\[x^{k+1}=g(x^k)\]
is ensured to converge to the fixed-point of $g$ by Banach's fixed-point theorem. Anderson acceleration ($\AA$)~\citemain{anderson1965iterative,Walker2011,Anderson2019} is a technique for speeding up the convergence of such an iterative process. Instead of using the update $x^{k+1}=g(x^k)$, it generates $x^{k+1}$ as an affine combination of the latest $m+1$ steps:
\begin{equation}\label{eq1-5}
		x^{k+1} = g(x^{k})+{\sum}_{i=1}^{m}\alpha_i^\ast (g(x^{k-i})-g(x^{k}))
\end{equation}
with the combination coefficients $\alpha^{\ast}=(\alpha^{\ast}_1, \ldots ,\alpha^{\ast}_m) \in\mathbb{R}^m$ being computed via an optimization problem
\begin{equation}\label{eq1-4}
\min_{\alpha}~\left\|f(x^{k})+{\sum}_{i=1}^{m}\alpha_i(f(x^{k-i})-f(x^{k}))\right\|^2,
\end{equation}
where $f(x^{k}) = g(x^{k}) - x^{k}$ denotes the \emph{residual function}.

$\AA$ was initially proposed to solve integral equations~\citemain{anderson1965iterative} and has gained popularity in recent years for accelerating fixed-point iterations~\citemain{Walker2011}. Examples include tensor decomposition~\citemain{Sterck2012}, linear system solving~\citemain{Pratapa2016}, and reinforcement learning~\citemain{geist2018anderson}, among many others~\citemain{matveev2018anderson,both2019anderson,an2017anderson,willert2015using,peng2018anderson,pavlov2018aa,pollock2019anderson,mai2019nonlinear,Zhang2019,mai2019anderson}.

On the theoretical side, it has been shown that $\AA$ is a quasi-Newton method for finding a root of the residual function~\citemain{Eyert1996,Fang2009,Rohwedder2011}. When applied to linear problems (i.e., if $g(x) = Ax - b$), $\AA$ is equivalent to the generalized minimal residual method (GMRES), \cite{potra2013characterization}. For nonlinear problems, $\AA$ is also closely related to the nonlinear generalized minimal residual method \cite{wang2021asymptotic}.
A local convergence analysis of $\AA$ for general nonlinear problems was first given in~\citemain{toth2015convergence,toth2017local} under the base assumptions that $g$ is Lipschitz continuously differentiable and the $\AA$ mixing coefficients $\alpha$, determined in \eqref{eq1-4}, stay in a compact set. However, the convergence rate provided in~\citemain{toth2015convergence,toth2017local} is not faster than the one of the original fixed-point iteration.
A more recent analysis in~\citemain{evans2020proof} shows that $\AA$ can indeed accelerate the local linear convergence of a fixed-point iteration up to an additional quadratic error term. This is further improved in \cite{pollock2021anderson} where q-linear convergence of $\AA$ is established. The convergence result in \cite{pollock2021anderson} requires sufficient linear independence of the columns of $[f(x^{k-1})-f(x^k),\dots,f(x^{k-m})-f(x^{k-m+1})]$ which is typically stronger than the previously mentioned boundedness assumption on the coefficients $\alpha$. By assuming the mixing coefficient $\alpha$ to be stationary during the iteration, an exact rate of $\AA$ is derived in \cite{wang2021asymptotic}. 

One issue of classical $\AA$ is that it can suffer from instability and stagnation~\citemain{potra2013characterization,scieur2016regularized}. Different techniques have been proposed to address this issue. For example, safeguarding checks were introduced in~\citemain{peng2018anderson,Zhang2019} to only accept an $\AA$ step if it meets certain criteria, but without a theoretical guarantee for convergence. Another direction is to introduce regularization to the problem~\eqref{eq1-4} for computing the combination coefficients. In~\citemain{fu2019anderson}, a quadratic regularization is used together with a safeguarding step to achieve global convergence of $\AA$ on Douglas-Rachford splitting, but there is no guarantee that the local convergence is faster than the original solver. In~\citemain{scieur2016regularized,Scieur2020}, a similar quadratic regularization is introduced to achieve local convergence, although no global convergence proof is provided. A more detailed discussion of related literature and specific techniques connected to our algorithmic design and development is deferred to Subsection~\ref{sec:ada-alg}.

As far as we are aware, none of the existing approaches and modified versions of $\AA$ guarantee both global convergence and accelerated local convergence. In this paper, we propose a novel $\AA$ globalization scheme that achieves these two goals simultaneously. Specifically, we apply a quadratic regularization with its weight adjusted automatically according to the effectiveness of the $\AA$ step. We adapt the nonmonotone trust-region framework in~\citemain{ulbrich2001nonmonotone} to update the weight and to determine the acceptance of the $\AA$ step. Our approach can not only achieve global convergence, but also attains the same local convergence rate established in~\citemain{evans2020proof} for $\AA$ without regularization. Furthermore, our local results also cover applications where the mapping $g$ is nonsmooth and differentiability is only required at a target fixed-point of $g$. To the best of our knowledge, this is the first globalization technique for $\AA$ that achieves the same local convergence rate as the original $\AA$ scheme. Numerical experiments on both smooth and nonsmooth problems verify the effectiveness and efficiency of our method. 
	
\textit{Notations}. Throughout this work, we restrict our discussion on the $n$-dimensional Euclidean space $\mathbb{R}^n$. For a vector $x$, $\|x\|$ denotes its Euclidean norm, and $\mathbb{B}_{\epsilon}(x) := \{y : \|y -x \| \leq \epsilon \}$ denotes the Euclidean ball centered at $x$ with radius $\epsilon$. For a matrix $A$, $\|A\|$ is the operator norm with respect to the Euclidean norm. We use $I$ to denote both the identity mapping (i.e., $I(x) = x$) and the identity matrix. For a function $h : \Rn \to \R^\ell$, the mapping $h^\prime : \Rn \to \R^{\ell \times n}$ represents its derivative. $h$ is called $L$-smooth if it is differentiable and $\|h^\prime(x)-h^\prime(y)\|\leq L\|x-y\|$ for all $x,y\in\Rn$. An operator $h:\Rn\to\Rn$ is called nonexpansive if for all $x,y\in\Rn$ we have
$\|h(x)-h(y)\| \leq \|x-y\|$.
We say that the operator $h$ is $\rho$-averaged for some $\rho\in(0,1)$ if there exists a nonexpansive operator $R : \Rn \to \Rn$ such that $h = (1-\rho)I+\rho R$. The set of fixed points of the mapping $h$ is defined via $\fix(h) := \{x: h(x)=x\}$. The interested reader is referred to~\citemain{BauCom11} for further details on operator theory.

The $\AA$ formulation in Eqs.~\eqref{eq1-5} and \eqref{eq1-4} assumes $k \geq m$. It can be adapted to account for the case  $k < m$ by using $\hat{m}$ coefficients instead where $\hat{m} = \min\{m, k\}$. Without loss of generality, we use $m$ to refer to the actual number of coefficients being used.  
\section{Algorithm and Convergence Analysis} \label{sec:alg} 

\subsection{Adaptive Regularization for $\AA$} \label{sec:ada-alg}
In the following, we set $f^k := f(x^k)$ and $g^k := g(x^k)$ to simplify notation.
We first note that the accelerated iterate computed via~\eqref{eq1-5} and \eqref{eq1-4} is invariant under permutations of the indices of $\{f^j\}$ and $\{g^j\}$. Concretely, let $\Pi_k := (k_0,k_1,\dots,k_m)$ be any permutation of the index sequence $(k,k-1,\dots,k-m)$. Then the point $x^{k+1}$ calculated in Eq.~\eqref{eq1-5}  also satisfies
\begin{equation} 
x^{k+1} = g^{k_0}+{\sum}_{i=1}^{m} \bar{\alpha}_i^k(g^{k_i}-g^{k_0}), 
\label{eq:PermutatedAA}
\end{equation}
with coefficients $\bar{\alpha}^k = (\bar{\alpha}_0^k, \ldots, \bar{\alpha}_m^k)$ computed via
\begin{equation}
\bar{\alpha}^k \in \argmin_{\alpha} \left\|f^{k_0}+{\sum}_{i=1}^{m}\alpha_i(f^{k_i}-f^{k_0})\right\|^2,
\label{eq:PermutatedAlphaProblem}
\end{equation}
which amounts to solving a linear system. In this paper, we use a particular class of permutations where 
\begin{equation}
k_0 =\max\left\{j \mid j \in \argmin\nolimits_{i\in\{k,k-1, \ldots, k-m\}}\|f^i\|\right\},
\label{eq:PermutationRule}
\end{equation}
i.e., $k_0$ is the largest index that attains the minimum residual norm among $\|f^{k-m}\|, \|f^{k-m+1}\|, \ldots, \|f^k\|$. 
As we will see later, this type of permutation allow us to apply certain nonmonotone globalization techniques and to ultimately establish local and global convergence of our approach. An ablation study on the potential effect of the permutation strategy is presented in Appendix~\ref{appx:AblationPermutation}.

One potential cause of instability of $\AA$ is the (near) linear dependence of the vectors $\{f^{k_i}-f^{k_0} : i = 1, \ldots, m\}$, which can result in (near) rank deficiency of the linear system matrix for the problem~\eqref{eq:PermutatedAlphaProblem}. To address this issue, we introduce a quadratic regularization to the  problem~\eqref{eq:PermutatedAlphaProblem} and compute the coefficients $\alpha^{k}$ via:
\begin{align}
\alpha^{k} = \argmin_{\alpha}~\|{\hat f}^k(\alpha)\|^2 + \lambda_k \|\alpha\|^2,
\label{eq:RegularizedOpt}
\end{align}
where $\lambda_k > 0$ is a regularization weight, and
\begin{equation}
{\hat f}^k(\alpha) := f^{k_0}+{\sum}_{i=1}^{m}\alpha_i(f^{k_i}-f^{k_0}).
\label{eq:Fhatk}
\end{equation}
The coefficients $\alpha^{k}$ are then used to compute a trial step in the same way as in Eq.~\eqref{eq:PermutatedAA}. In the following, we denote this trial step  as ${\hat g}^k(\alpha^k)$ where
\begin{equation}
	\hat{g}^k(\alpha) := g^{k_0}+{\sum}_{i=1}^{m}\alpha_i(g^{k_i}-g^{k_0}). 
	\label{eq:TrialStep}
\end{equation}
The trial step is accepted as the new iterate if it meets certain criteria (which we will develop in the following in detail).
Regularization such as the one in Eq.~\eqref{eq:RegularizedOpt} has been suggested in~\citemain{Anderson2019} and is applied in~\citemain{fu2019anderson,Scieur2020}. A major difference between our approach and the regularization in~\citemain{fu2019anderson,Scieur2020} is the choice of $\lambda_{k}$: in~\citemain{fu2019anderson} it is set in a heuristic manner, whereas in~\citemain{Scieur2020} it is either fixed or specified via grid search. We instead update $\lambda_k$ adaptively based on the effectiveness of the latest $\AA$ step. Specifically, we observe that a larger value of $\lambda_k$  can improve stability for the resulting linear system; it will also induce a stronger penalty for the magnitude of $\|\alpha^{k}\|$. In this case, the trial step $\hat{g}^k(\alpha^k)$ tends to be closer to $g^{k_0}$, 
which, according to Eq.~\eqref{eq:PermutationRule}, is the fixed-point iteration step with the smallest residual among the latest $m+1$ iterates.
On the other hand, a larger regularization weight may also hinder the fast convergence of $\AA$ if it is already effective in reducing the residual without regularization. Thus, $\lambda_k$ is dynamically adjusted according to the reduction of the residual in the current step.

Our adaptive regularization scheme is inspired by the similarity between the problem~\eqref{eq:RegularizedOpt} and the Levenberg-Marquardt (LM) algorithm~\citemain{Lev44,Mar63}, a popular approach for solving nonlinear least squares problems of the form $\min_x \|F(x)\|^2$, where $F$ is a vector-valued function. Each iteration of LM computes a variable update $d^k := x^{k+1} - x^k$ by solving a quadratic problem
\begin{equation*}
	\argmin_d ~\|F(x^k) + F'(x^k) d\|^2 + \bar{\lambda}_k \|d\|^2.
\end{equation*} 
Here, the first term is a local quadratic approximation of the target function $\|F(x)\|^2$ using the first-order Taylor expansion of $F$, while the second term is a regularization with a weight $\bar{\lambda}_k > 0$. LM can be considered as a regularized version of the classical Gauss-Newton (GN) method for nonlinear least squares optimization~\cite{madsen2004methods}. In GN, each iteration computes an initial step $d$ by minimizing the local quadratic approximation term only, i.e.,:
\begin{equation}
		\argmin_d ~\|F(x^k) + F'(x^k) d\|^2,
		\label{eq:GNProblem}
\end{equation}
which amounts to solving a linear system for $d$ with the positive semidefinite matrix $(F'(x^k))^T F'(x^k)$. Similar to $\AA$, the (near) linear dependence between the columns of $F'(x^k)$ can lead to (near) rank deficiency of the system matrix causing potential instability. To address this issue, LM introduces a quadratic regularization term for $d$, which adds a scaled identity matrix to the linear system matrix and prevents it from being singular.
Furthermore, LM measures the effectiveness of the computed update using a ratio of the resulting reduction of the target function and a predicted reduction based on the quadratic approximation. The measure is utilized to determine the acceptance of the update, to enforce monotonic decrease of the target function, and to update the regularization weight for the next iteration. Such an adaptive regularization is an instance of a trust-region method~\citemain{Conn2000}.

Taking a similar approach as LM, we define two functions $\ared{k}$ and $\pred{k}$ that measure the actual and predicted reduction of the residual resulting from the solution $\alpha^k$ to~\eqref{eq:RegularizedOpt}:
\begin{equation}
\ared{k} :=\res{k}-\|f(\hat{g}^k(\alpha^k))\|, \quad \pred{k} :=\res{k} -c\|\hat{f}^k(\alpha^k)\|,
\label{eq3-4}
\end{equation} 
where $c \in (0, 1)$ is a constant.
Here $\res{k}$ measures the residuals from the latest $m+1$ iterates via a convex combination:  
\begin{equation}
\res{k} := (1- m \gamma)\|f^{k_0}\| + {\sum}_{i=1}^m \gamma  \|f^{k_i}\|,
\label{eq:ReisdualCombination}
\end{equation}
with $\gamma \in (0, \frac{1}{m+1})$ such that a higher weight is assigned to the smallest residual $f^{k_0}$ among them.
Note that $\hat{g}^k(\alpha^k)$ is the trial step, and $f(\cdot)$ is the residual function. Thus $\ared{k}$ compares the latest residuals with the residual resulting from the trial step. This specific choice of $r_k$ is inspired by the local descent properties of $\AA$, see, e.g., \cite[Theorem 4.4]{evans2020proof}.
Moreover, note that $\hat{f}^k(\cdot)$ (see Eq.~\eqref{eq:Fhatk}) is a linear approximation of the residual function based on the latest residual values, and it is used in problem~\eqref{eq:RegularizedOpt} to derive the coefficients $\alpha^k$ for computing the trial step. Thus $\hat{f}^k(\alpha^k)$ is a predicted residual for the trial step based on the linear approximation, and $\pred{k}$ compares it with the latest residuals. The  constant $c$ guarantees that $\pred{k}$ has a positive value (as long as a solution to the problem has not been found; see Appendix~\ref{proof:PositivePredk} for a proof).
Similar to LM, we calculate the ratio
\begin{equation}
\rho_k = \frac{\ared{k}}{\pred{k}}
\label{eq:RatioRho}
\end{equation}
as a measure of effectiveness for the trial step $\hat{g}^k(\alpha^k)$ computed with Eqs.~\eqref{eq:RegularizedOpt} and \eqref{eq:TrialStep}. 
In particular, if $\rho_k \geq p_1$ with a threshold $p_1 \in (0,1)$, then from Eq.~\eqref{eq:RatioRho} and using the positivity of $\pred{k}$ we can bound the residual of $\hat{g}^k(\alpha^k)$ via
\begin{equation}
	\|f(\hat{g}^k(\alpha^k))\| \leq (1-p_1) r_k + p_1 c \|\hat{f}^k(\alpha^k)\|  < (1-p_1) r_k + p_1 \|f^{k_0}\|.
	\label{eq:NewCombination}
\end{equation}
Like $r_k$, the last expression $(1-p_1) r_k + p_1 \|f^{k_0}\|$ in Eq.~\eqref{eq:NewCombination} is also a convex combination of the latest $m+1$ residuals, but with a higher weight on the smallest residual $f^{k_0}$ than on $r_k$.
Hence, when $\rho_k \geq p_1$, we consider the decrease of the residual to be sufficient. In this case, we set $x^{k+1} = \hat{g}^k(\alpha^k)$ and say the iteration  is \emph{successful}. Otherwise, we discard the trial step and choose $x^{k+1} = g^{k_0} = g(x^{k_0})$, which corresponds to the fixed-point iteration step with the smallest residual among the latest $m+1$ iterates. 
Thus, by permuting the indices $(k,k-1,\ldots,k-m)$ according to $\Pi_k$, we can ensure to achieve the most progress in terms of reducing the residual when an $\AA$ trial step is rejected.

We also adjust the regularization weight $\lambda_k$ according to the ratio $\rho_k$. Specifically, we set
\begin{equation}
	\lambda_k= \mu_k \|f^{k_0}\|^2,
	\label{eq:LambdaK}
\end{equation}
where the factor $\mu_k > 0$ is automatically updated based on $\rho_{k}$ as follows:
\begin{itemize}
	\item If $\rho_{k} < p_1$, then we consider the decrease of the residual to be insufficient and we increase the factor in the next iteration via $\mu_{k+1} =  \eta_1\mu_k$ with a constant $\eta_1 > 1$. 
	\item If $\rho_{k} > p_2$ with a threshold $p_2 \in (p_1, 1)$, then we consider the decrease to be high enough and reduce the factor via $\mu_{k+1} =  \eta_2\mu_k$ with a constant $\eta_2 \in (0,1)$. This will relax the regularization so that the next trial step will tend to be closer to the original $\AA$ step.
	\item Otherwise, in the case $\rho_{k} \in [p_1, p_2]$, the factor remains the same in the next iteration.
\end{itemize} 
Here the choice of the parameters $p_1, p_2$ where $0 < p_1 < p_2 < 1$ follows the convention of basic trust-region methods~\cite{Conn2000}.

Our setting of $\lambda_k$ in Eq.~\eqref{eq:LambdaK} is inspired by~\citemain{fan2003modified} which relates the LM regularization weight to the residual norm. 
For our method, this setting ensures that the two target function terms in problem~\eqref{eq:RegularizedOpt} are of comparable scales, so that the adjustment of the factor $\mu_k$ is meaningful. This choice of $\lambda_k$ and the update rule of $\mu_k$ are quite standard in LM methods. However, the classical convergence analysis in~\citemain{fan2003modified} is not directly applicable here. In the LM method, the decrease of the residual can be predicted via its linearized model $\|F(x^k)+F'(x^k)d\|^2$. For $\AA$, the linearized residual $\hat f^k(\alpha^k)$ is not a model for the update $x^{k+1}=\hat g^k(\alpha^k)$ but for $\hat x^k(\alpha^k)$ instead. Since a linearized residual of $\hat g^k$ is not readily available, we use an upper bound for such a linearization of $\hat g^k(\alpha^k)$ which is exactly given by $c\|\hat f^k(\alpha^k)\|$. The whole method is summarized in Algorithm~\ref{algo1}.

\begin{algorithm}[t!]
	\caption{Anderson acceleration with adaptive regularization.}
	\label{algo1}
	\begin{algorithmic}[1]
		\REQUIRE  $x^0\in \mathbb{R}^{n}$, $\mu_0 > 0$, $k = 0$, $0<p_1<p_2<1,~0<\eta_2<1<\eta_1$, $0<c<1$.
		\FOR{$k = 0, 1, \ldots$}
		\STATE Compute $g^{k} = g (x^k)$, $f^{k} = g^k - x^k$.
		\IF{$\|f^k\|$ is smaller than a threshold $\epsilon_f$}
		\STATE Return $g^k$ as the result.
		\ENDIF
		\STATE Construct a permutation $\Pi_k = (k_0, k_1, \ldots, k_m)$ according to Eq.~\eqref{eq:PermutationRule}.
		\STATE Compute the regularization weight $\lambda_k$ via Eq.~\eqref{eq:LambdaK}.
		\STATE Solve the problem \eqref{eq:RegularizedOpt} to obtain the coefficients $\alpha^k$.
		\STATE Compute the ratio $\rho_k$ with Eq.~\eqref{eq:RatioRho}.
		\STATE Update the regularization factor $\mu_{k+1}$ via:
		$$\mu_{k+1}= \begin{cases} \eta_1\mu_k & \text{if}~\rho_k<p_1, \\ \mu_k & \text{if}~\rho_k\in[p_1,p_2], \\ \eta_2\mu_k & \text{if}~\rho_k>p_2. \end{cases}
		$$
		\IF{$\rho_k\geq p_1$} 
		\STATE {Set $x^{k+1}=\hat{g}^k(\alpha^k)$ via Eq.~\eqref{eq:TrialStep}.} 
		\ELSE 
		\STATE{Set $x^{k+1}=g^{k_0}$.} 
		\ENDIF	
		\ENDFOR
	\end{algorithmic}
\end{algorithm}

Unlike LM which enforces a monotonic decrease of the target function, our acceptance strategy allows the residual for $x^{k+1}$ to increase compared to the previous iterate $x^k$. Therefore, our scheme can be considered as a nonmonotone trust-region approach and follows the procedure investigated in~\citemain{ulbrich2001nonmonotone}. In the next subsections, we will see that this nonmonotone strategy allows us to establish unified global and local convergence results. In particular, besides global convergence guarantees, we can show transition to fast local convergence and an acceleration effect similar to the original $\AA$ scheme can be achieved.

The main computational overhead of our method lies in the optimization problem~\eqref{eq:RegularizedOpt}, which amounts to constructing and solving an $m \times m$ linear system
$( J^T J + {\lambda_k} I ) \alpha^k  =  -J^T f^{k_0}$
where $J = [ f^{k_1} - f^{k_0}, f^{k_2} - f^{k_0}, \ldots, f^{k_m} - f^{k_0} ] \in \mathbb{R}^{n \times m}$. 
A na\"{i}ve implementation that computes the matrix $J$ from scratch in each iteration will result in $O(m^2 n)$ time for setting up the linear system, whereas the system itself can be solved in $O(m^3)$ time. Since we typically have $m \ll n$, the linear system setup will become the dominant overhead.
To reduce the overhead, we note that each entry of $J^T J$ is a linear combination of inner products between $f^{k_0}, \ldots, f^{k_m}$. If we pre-compute and store these inner products, then it only requires additional $O(m^2)$ time to evaluate all entries. Moreover, the pre-computed inner products can be updated in $O(mn)$ time in each iteration, so we only need $O(mn)$ total time to evaluate $J^T J$. Similarly, we can evaluate $J^T f^{k_0}$ in $O(m)$ time. In this way, the linear system setup only requires $O(mn)$ time in each iteration. Moreover, as the parameter $m$ is often a small value independent of $n$ (and significantly smaller than $n$), the complexity $O(mn)$ is effectively linear with respect to $n$ and only incurs a small computational overhead.

\subsection{Global Convergence Analysis}
We now present our main assumptions on $g$ and $f$ that allow us to establish global convergence of Algorithm~\ref{algo1}. Our conditions are mainly based on a monotonicity property and on pointwise convergence of the iterated functions $g^{[k]} : \Rn \to \Rn$ defined as 
$g^{[k]}(x) := (\underbracket{\begin{minipage}[t][1.7ex][t]{9.7ex}\centering$\displaystyle g \circ \dots \circ g$\end{minipage}}_{k \; \text{times}})(x)$, for $k \in \N$.
\begin{assumption}
	\label{assump:GlobalConvergence}
	The functions $g$ and $f$ satisfy the following conditions:
	\begin{enumerate}[label=\textup{\textrm{(A.\arabic*)}},topsep=0pt,itemsep=0ex,partopsep=0ex,parsep=0ex,leftmargin=6ex]
	\item \label{Aone} $\|f(g(x))\|\leq \|f(x)\|$ for all $x \in \Rn$.
	\item \label{Atwo} $\lim\limits_{k\to \infty} \|f(g^{[k]}(x))\| = \nu$ for all $x \in \Rn$, where $\nu=\inf_{x\in\R^n}\|f(x)\|$.
	\end{enumerate}
	\end{assumption}
It is easy to see that Assumption~\ref{assump:GlobalConvergence} holds for any contractive function with $\nu = 0$. In particular, if $g$ satisfies \eqref{eq1-con}, we obtain 
\begin{align} \|f(g^{[k]}(x))\| & = \|g(g^{[k]}(x)) - g(g^{[k-1]}(x)) \| \label{eq2-conA} \\ \nonumber & \leq \kappa \|f(g^{[k-1]}(x))\| \leq \ldots \leq \kappa^{k} \|f(x)\| \to 0 \end{align}
as $k \to \infty$. In the following, we will verify that Assumption~\ref{assump:GlobalConvergence} also holds for $\rho$-averaged operators which define a broader class of mappings than contractions.
\begin{proposition} \label{prop:avop}
	Let $g:\Rn\to\Rn$ be a $\rho$-averaged operator with $\rho\in(0,1)$. 
	Then $g$ satisfies Assumption~\ref{assump:GlobalConvergence}. 
	\end{proposition}
	
	\begin{proof}
		By definition the $\rho$-averaged operator $g$ is also nonexpansive and thus, \ref{Aone} holds for $g$. To prove~\ref{Atwo}, let us set $y^0 := x$ and $y^{k+1} := g^{[k+1]}(x) = g(y^{k})$ for all $k$. By \refeq{Aone}, the sequence $\{\|f(y^{k})\|\}_k$ is monotonically decreasing. Therefore, we can assume that $\lim_{k\to\infty}\|f(y^{k})\|=\vartheta$. If $\vartheta>\nu$, then we may select $x^0\in \R^n$ such that $\|f(x^0)\|<\nu+\frac{1}{2}(\vartheta-\nu)$. Defining $x^{k+1}=g^{[k+1]}(x^0)=g(x^k)$ and applying \cite[Proposition 4.25(iii)]{BauCom11}, we have 
		$$\|x^{k+1}-y^{k+1}\|^2\leq  \|x^k-y^k\|^2 -\frac{1-\rho}{\rho}\|f(x^k)-f(y^k)\|^2.$$
		This yields 
		\[ \|x^{k+1}-y^{k+1}\|^2\leq  \|x^0-y^0\|^2 -\frac{1-\rho}{\rho}\sum_{i=0}^k\|f(x^i)-f(y^i)\|^2.         \]
		By the reverse triangle inequality and \ref{Aone}, we have
		\[  \| f(x^i)-f(y^i)\|\geq \|f(y^i)\|-\|f(x^i)\|\geq \vartheta-\|f(x^0)\|\geq \frac{1}{2}(\vartheta-\nu).           \]
		Combining with the previous inequality, we obtain
		\[   \|x^{k+1}-y^{k+1}\|^2\leq  \|x^0-y^0\|^2 -\frac{1-\rho}{2\rho}(k+1)(\vartheta-\nu).              \]
		Taking the limit $k\to\infty$, we reach a contradiction. So, we must have $\nu=\vartheta$, as desired.
		\qed
\end{proof}

\begin{remark} \label{rem:avop} Setting $\kappa$ (the Lipschitz constant of $g$) to $1$ in \eqref{eq2-conA}, we see that \ref{Aone} is always satisfied if $g$ is a nonexpansive operator. However, nonexpansiveness is not a necessary condition for \ref{Aone}.  In fact, we can construct an operator $g$ that is not nonexpansive but satisfies \ref{Aone} and \ref{Atwo}, e.g.,
$$ g : \R \to \R, \quad g(x) :=\left\{
\begin{aligned}
&0.5 x   & \text{if }x\in[0,1], \\
&0   & \text{otherwise.}  
\end{aligned}
\right.          $$
For any $x\in[0,1]$, we have $g^{[k]}(x) = 2^{-k}x$, $f(g^{[k]}(x)) = -2^{-(k+1)}x$ and it is not hard to verify \ref{Aone} and \ref{Atwo}. For any $x\notin [0,1]$ it follows $f(g(x))=f(0)=0$, thus \ref{Aone} and \ref{Atwo} also hold in this situation. However, since $g$ is not continuous, it can not be nonexpansive. 
\end{remark}

Because of Proposition~\ref{prop:avop}, our global convergence theory is applicable to a large class of iterative schemes. As an example, we show in the following that Assumption~\ref{assump:GlobalConvergence} is satisfied by forward-backward splitting, a popular optimization solver in machine learning.  

\begin{example} \label{ex:pgm}
Let us consider the nonsmooth optimization problem:
\begin{equation}
\min_{x \in \Rn}~r(x)+\varphi(x),
\label{eq:NonSmoothProblem}
\end{equation}
where both $r, \varphi : \Rn \to (-\infty,\infty]$ are proper, closed, and convex functions, and $r$ is  $L$-smooth. It is well known that $x^*$ is a solution to this problem if and only if it satisfies the nonsmooth equation:
\begin{equation*} \label{eq:opt-no} x^* - G_\mu(x^*) = 0, \quad  G_\mu(x) := \prox_{\mu\varphi}(x-\mu\nabla r(x)), \end{equation*}
where $\prox_{\mu\varphi}(x) := \argmin_y \varphi(y) + \frac{1}{2\mu} \|x-y\|^2$, $\mu > 0$, is the proximity operator of $\varphi$, see also Corollary 26.3 of \citemain{BauCom11}. We can then compute $x^*$ via the iterative scheme
\begin{equation}
	x^{k+1} = G_\mu(x^k).
	\label{eq:ForwardBackward}
\end{equation}
$G_\mu$ is known as the forward-backward splitting operator and it is a $\rho$-averaged operator for all $\mu\in(0,\frac{2}{L})$, see \citemain{byrne2014elementary}. Hence, Assumption~\ref{assump:GlobalConvergence} holds and our theory can be used to study the global convergence of Algorithm~\ref{algo1} applied to \eqref{eq:ForwardBackward}.    
\end{example}

\begin{remark}
	For problem~\eqref{eq:NonSmoothProblem}, it can be shown that Douglas-Rachford splitting, as well as its equivalent form of ADMM, can both be written as a $\rho$-averaged operator with $\rho \in (0,1)$ (see, e.g.,~\citemain{liang2016local}). Therefore, the applications considered in~\citemain{fu2019anderson} are also covered by Assumption~\ref{assump:GlobalConvergence}.
\end{remark}

We can now show the global convergence of Algorithm~\ref{algo1}: 
\begin{theorem}
\label{prop3-3}
Suppose Assumption~\ref{assump:GlobalConvergence} is satisfied and let $\{x^k\}$ be generated by Algorithm~\ref{algo1} with $\epsilon_f = 0$. Then
$$   \lim_{k\rightarrow\infty}\|f^k\|=\nu,          $$
where $\nu=\inf_{x\in\R^n}\|f(x)\|$.
\end{theorem}

\begin{proof}
	In the following, we will use $\mathcal S$ to denote the set of indices for all successful iterations, i.e., $\mathcal S := \{k: \rho_k \geq p_1\}$. 
	To simplify the notation, we introduce a function $\mathcal P:\mathbb{N}\to\mathbb{N}$ defined as
	$$  \mathcal P(k):=\max \left\{j \mid j \in{\argmin}_{i\in\{k,k-1,\ldots,k-{{m}} \}}\|f^i\|\right\}. $$
	Notice that the number $\mathcal P(k)$ coincides with $k_0$ for fixed $k$. 
	
	If Algorithm~\ref{algo1} terminates after a finite number of steps, the conclusion simply follows from the stopping criterion. Therefore, in the following, we assume that a sequence of iterates of infinite length is generated. We consider two different cases:
	\begin{description}[topsep=1ex,itemsep=1ex]
	\item[Case~1:] $|\mathcal{S}|<\infty$. Let $\bar{k}$ denote the index of the last successful iteration in $\mathcal{S}$ (we set $\bar k = 0$ if $\mathcal{S}=\emptyset$). We first show that $\mathcal{P}(k)=k$ for all $k\geq\bar k+1$. Due to $\bar k+1\notin\mathcal{S}$, it follows $x^{\bar k+1}=g(x^{\mathcal P(\bar k)})$ and by \ref{Aone}, this implies $\|f(x^{\bar k+1})\|\leq \|f(x^{\mathcal P(\bar k)})\|$. From the definition of $\mathcal P$, we have $\|f(x^{\mathcal P(\bar k)})\|\leq \|f^{\bar k-i}\|$ for every $0\leq i\leq \min\{m,\bar k\}$ and hence $\mathcal P(\bar k+1)=\bar k+1$. An inductive argument then yields $\mathcal{P}(k)=k$ for all $k\geq \bar k +1$. Notice that for any $k\geq \bar k+1$, we have $k\notin\mathcal S$ and $x^{k+1}=g(x^{\mathcal P(k)})=g(x^k)$. Utilizing (A.2), it follows that $\|f^k\| = \|f(g^{[k-\bar k]}(x^{\mathcal P(\bar k)}))\| \to \nu$ as $k \to \infty$.
	
	\item[Case~2:] $|\mathcal{S}|=\infty$. Let us denote 
	\begin{equation*}
		W_k:=\max_{k-m\leq i\leq k}\|f^i\|. 
	\end{equation*}
	We first show that the sequence $\{W_k\}$ is non-increasing.
	\begin{itemize}
		\item If $k\in\mathcal{S}$, then we have:
		\begin{align*}
			& p_1 \leq \rho_{k} = \frac{\ared{k}}{ \pred{k}}.
		\end{align*}
		We already know from Appendix~\ref{proof:PositivePredk} that $\pred{k} > 0$.
		Since $p_1 > 0$, it also holds that $\ared{k} >0$. Hence, if $c \|\hat f^k(\alpha^k)\| \leq \|f^{k+1}\|$, then using $\res{k} \leq W_k$ and \eqref{eq:esti-app-res} from Appendix~\ref{proof:PositivePredk}, we can derive:
			\begin{align*}
				p_1 & \leq \frac{\ared{k}}{ \pred{k}} 
				= 1 + \frac{c \|\hat f^k(\alpha^k)\| - \|f^{k+1}\|}{\pred{k}}= 1 + \frac{c \|\hat f^k(\alpha^k)\| - \|f^{k+1}\|}{r_k-c\|\hat f^k(\alpha^k)\|}  \\
				&\leq 1 + \frac{c \|\hat f^k(\alpha^k)\| - \|f^{k+1}\|}{W_k-c\|\hat f^k(\alpha^k)\|}=\frac{W_k-\|f^{k+1}\|}{W_k-c\|\hat f^k(\alpha^k)\|} 
				\leq \frac{W_k-\|f^{k+1}\|}{(1-c)W_k},
			\end{align*}
		which implies
		\begin{equation}
			\|f^{k+1}\|\leq c_p W_k,
			\label{eq:WkBound1} 
		\end{equation}
		where $c_p := 1-(1-c)p_1 < 1$. Otherwise, if $c \|\hat f^k(\alpha^k)\| > \|f^{k+1}\|$, then we have 
		\begin{equation}
			\|f^{k+1}\| \leq c W_k \leq c_p W_k.
			\label{eq:WkBound2}
		\end{equation}
		
		\item If $k \notin \mathcal S$, we have $x^{k+1}=g^{\mathcal P(k)}$. By Assumption~\ref{Aone}, it then follows that
		\begin{equation}
			\|f^{k+1}\|\leq \|f^{\mathcal P(k)}\|\leq W_{k}.
			\label{eq:WkBound3}
		\end{equation}
	\end{itemize}
	Eqs.~\eqref{eq:WkBound1}, \eqref{eq:WkBound2} and \eqref{eq:WkBound3} show that $\|f^{k+1}\| \leq W_{k}$. By definition of $W_{k}$, we then have $W_{k+1} \leq \max\{\|f^{k+1}\|, W_k\}  = W_k$. This shows that the sequence $\{W_k\}$ is non-increasing. Next, we verify 
	$$W_{k+m+1}\leq c_p W_k$$ 
	for all $k \in \mathcal S$. It suffices to prove that for any $i$ satisfying $k+1\leq i\leq k+m+1$, we have $\|f^i\|\leq c_p W_{k}$. Since we consider a successful iteration $k\in\mathcal{S}$, our previous discussion has already shown that $\|f^{k+1}\|\leq c_p W_k$. We now assume $\|f^i\|\leq c_pW_k$ for some $k+1\leq i\leq k+m$. If $i\in\mathcal S$, we obtain $\|f^{i+1}\|\leq c_p W_i \leq c_p W_k$. Otherwise, it follows that $\|f^{i+1}\|\leq \|f^{\mathcal P(i)}\| \leq \|f^i\|\leq c_p W_k$. Hence, by induction, we have $W_{k+m+1}\leq c_p W_k$ for all $k \in \mathcal S$. Since $\{W_k\}$ is non-increasing and we assumed $|\mathcal S|=\infty$, this establishes $W_k\rightarrow 0$ and $\|f^k\|\to 0$. In this case, we can infer $\nu=0$ and the proof is complete.
	\end{description}
	\qed
\end{proof}

\begin{remark}
This global result does not depend on the specific update rule for the regularization weight $\lambda_k$. Indeed, global convergence mainly results from our acceptance mechanism and hence, as a consequence of our proof, different update strategies for $\lambda_k$ can also be applied. Our choice of $\lambda_k$ in \eqref{eq:LambdaK}, however, will be essential for establishing local convergence of the method. 
\end{remark}

\subsection{Local Convergence Analysis} \label{sec:loc}

Next, we analyze the local convergence of our proposed approach, starting with several assumptions.
\begin{assumption}
\label{assum3-1}
The function $g : \Rn \to \Rn$ satisfies the following conditions:
\begin{enumerate}[label=\textup{\textrm{(B.\arabic*)}},leftmargin=8ex]
  \item\label{Bone} $g$ is Lipschitz continuous with a constant $\kappa<1$.
  \item\label{Btwo} $g$ is differentiable at $x^*$ where $x^*$ is the fixed point of the mapping $g$.
\end{enumerate}
\end{assumption}

\begin{remark}
\ref{Bone} is a standard assumption widely used in the local convergence analysis of $\AA$~\citemain{toth2015convergence,evans2020proof,scieur2016regularized,Scieur2020}. 
The existing analyses typically rely on the smoothness of $g$. In contrast, \ref{Btwo} allows $g$ to be nonsmooth and only requires it to be differentiable at the fixed point $x^*$, allowing our assumptions to cover a wider variety of methodologies such as forward-backward splitting and Douglas-Rachford splitting under appropriate assumptions, see Appendix~\ref{app:Verification}. We note that in \cite{bian2021anderson} the Lipschitz differentiability of $g$ is replaced by continuous differentiability around $x^*$, while we only assume differentiability at one point. This technical difference is based on the observation that an expansion of the residual $f^k$ is only required at the point $x^*$ and not at the iterates $x^k$ which allows to work with weaker differentiability requirements. We further note that $\AA$ has been investigated for nonsmooth $g$  in \citemain{zhang2018globally,fu2019anderson} but without local convergence analysis. Recent convergence results of $\AA$ for a scheme related to the proximal gradient method discussed in Example~\ref{ex:pgm} can also be found in~\citeappx{mai2019anderson}. While the local assumptions and derived convergence rates in~\citeappx{mai2019anderson} are somewhat similar to our local results, we want to highlight that the algorithm and analysis in~\citeappx{mai2019anderson}  are tailored to convex composite problems of the form \eqref{eq:NonSmoothProblem}. Moreover, the global results in~\citeappx{mai2019anderson} are shown for a second, guarded version of $\AA$ and are based on the strong convexity of the problem. In contrast and under conditions that are not stronger than the local assumptions in~\citeappx{mai2019anderson}, we will establish unified global-local convergence of our approach for general contractions. In Section~\ref{sec:Results}, we verify the conditions \ref{Bone} and \ref{Btwo} on the numerical examples, with a more detailed discussion in Appendix~\ref{app:Verification}.
\end{remark}

\begin{remark} \label{rem:this}
	\ref{Bone} implies that the function $g$ is contractive, which is a sufficient condition for \ref{Aone} and \ref{Atwo}. Thus, a function $g$ satisfying Assumption~\ref{assum3-1} will also fulfill Assumption~\ref{assump:GlobalConvergence} with $\nu = 0$.
\end{remark}

Similar to the local convergence analyses in~\citemain{toth2015convergence,evans2020proof}, we also work with the following condition:
\begin{assumption}
\label{assum3-4}
For the solution $\bar{\alpha}^k$ to the unregularized $\AA$ problem~\eqref{eq:PermutatedAlphaProblem}, 
there exists $M > 0$ such that $\|\bar{\alpha}^k\|_{\infty}\leq M$ for all $k$ sufficiently large.
\end{assumption}
\begin{remark}
The assumptions given in \citemain{toth2015convergence,evans2020proof} are formulated without permuting the last $m+1$ indices. We further note that we do not require the solution $\bar \alpha^k$ to be unique.
\end{remark}

The acceleration effect of the original $\AA$ scheme has only been studied very recently in \citemain{evans2020proof} based on slightly stronger assumptions. In particular, their result can be stated as
\begin{align}
\label{eq:Acceleration}
\|f(\hat{g}^k(\bar{\alpha}^k))\|\leq \kappa\theta_k\|f^{k_0}\|+\sum\nolimits_{i=0}^mO(\|f^{k-i}\|^2),
\end{align}
where 
\begin{equation} \label{eq:theta-k}   \theta_k:={\| \hat{f}^k(\bar{\alpha}^k)\|}/{\|f^{k_0}\|} \end{equation}
is an acceleration factor. 
Since $\bar{\alpha}^k$ is a solution to the problem~\eqref{eq:PermutatedAlphaProblem}, 
we have $\| \hat{f}^k(\bar{\alpha}^k)\| \leq \| \hat{f}^k(0)\| = \|f^{k_0}\|$ so that $\theta_k \in [0,1]$. Then~\eqref{eq:Acceleration} implies that for a fixed-point iteration that converges linearly with a contraction constant $\kappa$, $\AA$ can improve the convergence rate locally.
In the following, we will show that our globalized $\AA$ method possesses similar characteristics under weaker assumptions. 
 
We first verify that after finitely many iterations, every step $x^{k+1} = \hat g^k(\alpha^k)$ is accepted as a new iterate. Thus, our method eventually reduces to a pure regularized $\AA$ scheme.

\begin{theorem}
\label{lemma3-10}
Suppose that Assumptions~\ref{assum3-1} and~\ref{assum3-4} hold and let the constant $c$ in \eqref{eq3-4} be chosen such that $c\geq\kappa$. Then, the sequence $\{x^k\}$ generated by Algorithm~\ref{algo1} (with $\epsilon_f = 0$) either terminates after finitely many steps, or converges to the fixed point $x^*$ and there exists some $\ell\in\mathbb{N}$ such that $\rho_k\geq p_2$ for all $k\geq \ell$. In particular, every iteration $k \geq \ell$ is successful with $x^{k+1}=\hat{g}^k(\alpha^k)$. 
\end{theorem}
\begin{proof}
	Our proof consists of three steps. We first show the convergence of the whole sequence $\{x^k\}$ to the fixed point $x^*$. Afterwards we derive a bound for the residual $\|f(\hat g^k(\alpha^k))\|$ that can be used to estimate the actual reduction $\ared{k}$. In the third step, we combine our observations to prove the transition to the full $\AA$ method, i.e., we show that there exists some $\ell$ with $k\in\mathcal{S}$ for all $k\geq \ell$. 
	\begin{description}[topsep=1ex,itemsep=1ex]
	\item[Step 1:] Convergence of $\{x^k\}$. By~\ref{Bone}, $g$ is a contraction, i.e., for all $x \in \Rn$ we have
	\begin{equation} \|x - x^*\| \leq \|x-g(x)\| + \|g(x)-g(x^*)\| \leq \|f(x)\| + \kappa \|x-x^*\| \label{eq:con-lw} \end{equation}
	and it follows $\|f^k\| = \|f(x^k)\|\geq (1-\kappa)\|x^k-x^*\|$ for all $k$. Theorem~\ref{prop3-3} and Remark~\ref{rem:this} guarantee $\lim_{k\rightarrow\infty}\|f^k\|=0$ and hence, we can infer $x^k \to x^*$. 
	
	\item[Step 2:] Bounding $\|f(\hat g^k(\alpha^k))\|$. Introducing 
	\[ \hat{x}^k(\alpha^k) := x^{k_0} + {\sum}_{i=1}^m \alpha_i^k (x^{k_i} - x^{k_0}) \]
	and using~\ref{Bone}, we can bound the residual $\|f(\hat g^k(\alpha^k))\|$ directly as follows:
	\begin{align*}
		\notag \|f(\hat g^k(\alpha^k))\| & = \|g(\hat g^k(\alpha^k))-\hat g^k(\alpha^k)\|  \\
		&  \leq \|g(\hat g^k(\alpha^k))-g(\hat{x}^k(\alpha^k))\| +\|g(\hat{x}^k(\alpha^k))-\hat g^k(\alpha^k)\| \notag \\ 
		&  \leq \kappa \|\hat g^k(\alpha^k) - \hat{x}^k(\alpha^k)\| +\|g(\hat{x}^k(\alpha^k))-\hat g^k(\alpha^k)\| \notag \\
		&  = \kappa \|\hat f^k(\alpha^k)\| +\|g(\hat{x}^k(\alpha^k))-\hat g^k(\alpha^k)\|.
	\end{align*}
	We now continue to estimate the second term $\|g(\hat{x}^k(\alpha^k))-\hat g^k(\alpha^k)\|$. From the algorithmic construction and the definition of $\alpha^k$ and $\bar \alpha^k$, it follows that 
	\begin{equation} 
		\|\hat f^k(\alpha^k)\|^2 + \lambda_k \|\alpha^k\|^2 
		\leq \|\hat f^k(\bar\alpha^k)\|^2 + \lambda_k \|\bar\alpha^k\|^2 \leq \|\hat f^k(\alpha^k)\|^2 + \lambda_k \|\bar\alpha^k\|^2, \label{eq:alpha-k}
	\end{equation}
	which implies $\|\alpha^k\|_{\infty}\leq\|\alpha^k\|\leq\|\bar{\alpha}^k\|\leq \sqrt{m}\|\bar{\alpha}^k\|_\infty \leq \sqrt{m}M$ for all $k$. Defining $\nu^k=(\nu^k_0,\dots,\nu^k_m)\in\mathbb{R}^{m+1}$ with $\nu^k_0=1-\sum_{i=1}^m\alpha^k_i$ and $\nu^k_j = \alpha^k_j$ for $1\leq j \leq m$, we obtain
	\begin{align*}
		\hat g^k(\alpha) &= {\sum}_{i=0}^m\nu_i^k g(x^{k_i}), \quad \hat x^k(\alpha) = {\sum}_{i=0}^m\nu_i^k x^{k_i}, 
	\end{align*}
	$\|\nu^k\|_\infty \leq 1+m^{\frac32}M$, and $\sum_{i=0}^m \nu^k_i = 1$.
	Consequently, applying the estimate \eqref{eq:con-lw} derived in step 1, it follows
	\begin{align*}   
		\|\hat{x}^k(\alpha^k)-x^*\|  = \left \| {\sum}_{i=0}^m \nu_i^k (x^{k_i}-x^*) \right \|  & \leq (1+ m^{\frac32}M) {\sum}_{i=0}^m\|x^{k_i}-x^*\| \\ & \leq (1+ m^{\frac32}M)(1-\kappa)^{-1} {\sum}_{i=0}^m\|f^{k-i}\|
	\end{align*}
	which shows $\hat x^k(\alpha^k) \to x^*$ as $k \to \infty$. This also establishes
	\begin{equation} o(\|\hat{x}^k(\alpha^k)-x^*\|) = o\left({\sum}_{i=0}^m\|f^{k-i}\|\right) \quad k \to \infty. \label{eq:oo} \end{equation}
	Note that the differentiability of $g$ at $x^*$ -- as stated in Assumption~\ref{Btwo} -- implies $\|g(y) - g(x^*) - g^\prime(x^*)(y-x^*)\| = o(\|y-x^*\|)$ as $y \to x^*$. Applying this condition to different choices of $y$ and the boundedness of $\nu^k$, we can obtain
\begingroup
\allowdisplaybreaks
	\begin{align}
		& \hspace{-6ex} \|g(\hat x^k(\alpha^k)) - \hat g^k(\alpha^k) \| \notag\\ 
		= ~&\|g(\hat{x}^k(\alpha^k))-g(x^*) + g(x^*) - \hat g^k(\alpha^k)\| \notag\\ 
		\leq ~& \|g'(x^*)(\hat x^k(\alpha^k)-x^*) + g(x^*) - \hat g^k(\alpha^k) \|  +o(\|\hat x^k(\alpha^k)-x^*\|) \notag\\
		= ~& \left\| {\sum}_{i=0}^m \nu^k_i [g'(x^*)(x^{k_i}-x^*) + g(x^*) - g(x^{k_i})] \right\|  +o(\|\hat x^k(\alpha^k)-x^*\|) \notag\\
		\leq ~& {\sum}_{i=0}^m o(\|x^{k_i}-x^*\|) +o(\|\hat x^k(\alpha^k)-x^*\|) \leq o\left({\sum}_{i=0}^m\|f^{k-i}\|\right). \label{eq3-7}
	\end{align}
\endgroup
	Here, we also used \eqref{eq:con-lw}, \eqref{eq:oo}, and $\sum_{i=0}^m o(\|f^{k-i}\|) = o\left({\sum}_{i=0}^m\|f^{k-i}\|\right)$. Combining our results, this yields
	\begin{equation} \|f(\hat g^k(\alpha^k))\| \leq \kappa \|\hat f^k(\alpha^k)\| + o\left({\sum}_{i=0}^m\|f^{k-i}\|\right) \quad k \to \infty. \label{eq:res-esti} \end{equation}

	\item[Step 3:] Transition to fast local convergence. As in the proof of Theorem~\ref{prop3-3}, let us introduce 
	\[
	W_k := \max_{k-m\leq i\leq k}\|f^i\|. 
	\]
	Due to \eqref{eq:res-esti} and $o\left({\sum}_{i=0}^m\|f^{k-i}\|\right) = o(W_k)$ there then exists $\ell \in \N$ such that
	\[  \|f(\hat g^k(\alpha^k))\| \leq \kappa \|\hat f^k(\alpha^k)\| + (1-p_2)\min\{\gamma,1-c\}W_k \]
	for all $k \geq \ell$.  
	Hence, using $c\geq \kappa$, we have
	\begin{align*}
		\ared{k} & = \res{k} - \|f(\hat{g}^k(\alpha^k))\| 
		\geq \pred{k}- (1-p_2)\min\{\gamma,1-c\}W_k.
	\end{align*}
	Similarly, for the predicted reduction $\pred{k}$ we can show 
	\begin{align*}
		\pred{k} &= (1-\gamma m) \|f^{k_0}\| + \gamma {\sum}_{i=1}^{m} \|f^{k_i}\|-c\|\hat{f}^k(\alpha^k)\| \\ & \geq \gamma {\sum}_{i=0}^{m} \|f^{k_i}\| + (1-\gamma(m+1)) \|f^{k_0}\| - c \|f^{k_0}\| \\
		&\geq \gamma W_k+(1-\gamma)\|f^{k_0}\|-c\|f^{k_0}\|. 
	\end{align*}
	Thus, if $1-\gamma-c \geq 0$, we obtain $\pred{k} \geq \gamma W_k$. Otherwise, it follows $\pred{k} \geq (1-c) W_k$ and together this yields $\pred{k} \geq \min\{\gamma,1-c\}W_k$. Combining the last estimates, we can finally deduce
	$$ \frac{\ared{k}}{ \pred{k}}\geq \frac{\pred{k}- (1-p_2)\min\{\gamma,1-c\}W_k}{\pred{k}} \geq p_2,  $$
	which completes the proof.
	\end{description}
	\qed
\end{proof}
\begin{remark}
Our novel nonmonotone acceptance mechanism is the central component of our proof for Theorem~\ref{lemma3-10}, as it allows us to balance the additional error terms caused by an $\AA$ step.
\end{remark}

Next, we show that our approach can enhance the convergence of the underlying fixed-point iteration and that it has a local convergence rate similar to the original $\AA$ method as given in~\citemain{evans2020proof}.
\begin{theorem}
\label{thm:LocalConvergenceRate}
Suppose that Assumptions~\ref{assum3-1}, and~\ref{assum3-4} hold and let the parameters $c, \epsilon_{f}$ in Algorithm~\ref{algo1} satisfy $c\geq \kappa$ and $\epsilon_{f}=0$. Then, for $k \to \infty$ it holds that:
\begin{align*}
\|f^{k+1}\|\leq \kappa\theta_k\|f^{k_0}\|+o\left(\sum\nolimits_{i=0}^{m}\|f^{k-i}\|\right),
\end{align*}
where $\theta_k:=\|\hat{f}^k(\bar{\alpha}^k)\|/\|f^{k_0}\|$ is the corresponding acceleration factor. In addition, the sequence of residuals $\{\|f^k\|\}$ converges r-linearly to zero with a rate arbitrarily close to $\kappa$, i.e., for every $\eta \in (\kappa,1)$ there exist $C > 0$ and $\hat\ell \in \N$ such that
\[ \|f^k\| \leq C \eta^k \quad \forall~k \geq \hat\ell. \]
\end{theorem}

\begin{proof}
	Theorem~\ref{lemma3-10} implies $\rho_k\geq p_2$ for all $k \geq \ell$ and hence, from the update rule of Algorithm~\ref{algo1}, it follows that $$\mu_k=\eta_2\mu_{k-1} \quad \forall~k \geq \ell.$$ 
	Then by \eqref{eq:LambdaK}, we can infer $\lambda_k=o(\|f^{k_0}\|^2)$. Using Eq.~\eqref{eq:alpha-k} and Assumption~\ref{assum3-4}, this shows 
	\begin{equation} \|\hat{f}^k(\alpha^k)\|  \leq \|\hat{f}^k(\bar{\alpha}^k)\|+\sqrt{\lambda_k }\|\bar{\alpha}^k\| = \|\hat{f}^k(\bar{\alpha}^k)\|+o(\|f^{k_0}\|). \label{eq:app-e-esti} \end{equation}
	Thus, by \eqref{eq:res-esti}, we obtain
	\begin{align}
		\nonumber \|f^{k+1}\| \leq~ &\kappa\|\hat{f}^{k}(\alpha^k)\|+ o\left({\sum}_{i=0}^m \|f^{k-i}\|\right)\\ 
		\nonumber \leq~& \kappa \|\hat{f}^{k}(\bar{\alpha}^k)\|+o(\|f^{k_0}\|)+o\left({\sum}_{i=0}^m\|f^{k-i}\|\right) \\ 
		=~ & \kappa \theta_k \|f^{k_0}\|+ o\left({\sum}_{i=0}^m\|f^{k-i}\|\right), \label{eq:ineq-pp}
	\end{align}
	as desired. In order to establish r-linear convergence, we follow the strategy presented in \citeappx{toth2015convergence}. Let $\eta \in (\kappa,1)$ be a given rate. Then, due to $\|f^k\| \to 0$ and using \eqref{eq:ineq-pp}, there exists $\hat\ell \in \N$ such that 
	\begin{equation} \label{eq:ineq-pp2} \|f^{k+1}\| \leq \kappa\|{f}^{k_0}\|+ \bar \nu \cdot {\sum}_{i=0}^m \|f^{k-i}\| \end{equation}
	for all $k \geq \hat\ell$ where $\bar \nu := \frac{1-\eta}{1-\eta^{m+1}} \eta^m (\eta-\kappa)$. Defining $C := \eta^{-\hat\ell}  \max_{\hat\ell-m\leq i \leq \hat\ell} \|f^i\| = W_{\hat\ell}\,\eta^{-\hat\ell}$, we then have 
	\[ \|f^j\| \leq W_{\hat\ell} = (W_{\hat\ell}\,\eta^{-j}) \eta^j \leq (W_{\hat\ell}\,\eta^{-\hat\ell}) \eta^j = C \eta^j. \]
	for all $\hat\ell-m \leq j \leq \hat\ell$. We now claim that the statement $\|f^{k}\| \leq C \eta^k$ holds for all $k \geq \hat\ell$. As just shown, this is obviously satisfied for the base case $k = \hat\ell$. As part of the inductive step, let us assume that the estimate $\|f^{j}\| \leq C \eta^j$ holds for all $j = \hat\ell, \hat\ell+1,\ldots,k$. (In fact, this bound also holds for $j = \hat\ell-m,\ldots,\hat\ell-1$). By the definition of the index $k_0$, we have $\|f^{k_0}\| \leq C\eta^k$ and, due to \eqref{eq:ineq-pp2}, it follows
	\begin{align*} \|f^{k+1}\| & \leq \kappa\|{f}^{k_0}\|+ \bar \nu \cdot {\sum}_{i=0}^m \|f^{k-i}\|  \leq C\kappa\eta^k + C\bar \nu \eta^k {\sum}_{i=0}^m \left(\frac{1}{\eta}\right)^i \\ & = C\eta^k \left[ \kappa + \bar \nu \cdot  \frac{1-\eta^{-(m+1)}}{1-\eta^{-1}} \right]  = C\eta^k \left[ \kappa + \frac{\bar \nu}{\eta^m} \cdot  \frac{1-\eta^{m+1}}{1-\eta} \right] = C\eta^{k+1}. \end{align*} 
	Hence, our claim also holds for $k+1$ which finishes the induction and proof. 
	\qed
\end{proof}
	
Under a stronger differentiability condition and stricter update rule for $\lambda_k$, we can recover the same local rate as in~\citemain{evans2020proof}: 
 
\begin{corollary} \label{cor:stronger-loc} Let the assumptions stated in Theorem~\ref{thm:LocalConvergenceRate} hold and let $g$ satisfy the differentiability condition
%
\[ \|g(x) - g(x^*) - g^\prime(x^*)(x-x^*)\| = O(\|x-x^*\|^2) \;\; \text{as} \;\; x \to x^*. \]
Suppose that the weight $\lambda_k$ is updated via $\lambda_k = \mu_k \|f^{k_0}\|^4$. Then, for all $k$ sufficiently large we have 
\begin{align*}
\|f^{k+1}\|\leq \kappa\theta_k\|f^{k_0}\|+ \sum\nolimits_{i=0}^{m} O(\|f^{k-i}\|^2).
\end{align*}
\end{corollary}

\begin{proof} As mentioned in the remark after Theorem~\ref{prop3-3}, our global results do still hold if a different update strategy is used for the weight parameter $\lambda_k$. Moreover, the proof of Theorem~\ref{lemma3-10} does also not depend on the specific choice of $\lambda_k$. Consequently, we only need to improve the bound \eqref{eq:res-esti} for $\|f(\hat g^k(\alpha^k))\|$ derived in step 2 of the proof of Theorem~\ref{lemma3-10}. Using the additional differentiability property $\|g(y) - g(x^*) - g^\prime(x^*)(y-x^*)\| = O(\|y-x^*\|^2)$, $y \to x^*$, we can directly improve the estimate for $\|g(\hat x^k(\alpha^k)) - \hat g^k(\alpha^k)\|$ in \eqref{eq3-7} as follows: 
	\begin{align*} \|g(\hat x^k(\alpha^k)) - \hat g^k(\alpha^k)\| \leq {\sum}_{i=0}^m O(\|x^{k_i}-x^*\|^2) + O(\|\hat x^k(\alpha^k)-x^*\|^2). \end{align*}
	Using the bound $(\sum_{i=0}^m y_i)^2 \leq (m+1) \sum_{i=0}^m y_i^2$ for $y \in \R^{m+1}$, we obtain $\|\hat x^k(\alpha^k)-x^*\|^2 = \sum_{i=0}^m O(\|f^{k-i}\|^2)$ and thus, mimicking and combining the derivations in step 2 of the proof of Theorem~\ref{lemma3-10}, we have
	\[ \|g(\hat x^k(\alpha^k)) - \hat g^k(\alpha^k)\| \leq {\sum}_{i=0}^m O(\|f^{k-i}\|^2) \] 
	and
	\begin{equation} \label{eq:eq-pp-3} \|f(\hat g^k(\alpha^k))\| \leq \kappa \|\hat f^k(\alpha^k)\| + {\sum}_{i=0}^m O(\|f^{k-i}\|^2) \end{equation}
	as $k \to \infty$. As in the previous proof, we can now infer $\mu_k \to 0$ (this follows from $\rho_k \geq p_2$ for all $k$ sufficiently large) and $\lambda_k = o(\|f^{k_0}\|^4)$. Furthermore, as in \eqref{eq:app-e-esti}, due to Eq.~\eqref{eq:alpha-k} and Assumption~\ref{assum3-4}, it holds that $\|\hat f^k(\alpha^k)\| \leq \|\hat f^k(\bar\alpha^k)\| + o(\|f^{k_0}\|^2)$. Combining this result with \eqref{eq:eq-pp-3}, we can then establish the convergence rate stated in Corollary~\ref{cor:stronger-loc}. 
\qed
\end{proof}

\begin{remark}
The stronger differentiability condition, which was also used in~\citemain{evans2020proof} and other local analyses, is, e.g., satisfied when the derivative $g^\prime$ is locally Lipschitz continuous around $x^*$. More discussions of this property can also be found in Appendix~\ref{app:Verification}. We note that under this type of stronger differentiability, we can only improve the order of the remainder linearization error terms and not the linear rate of convergence.
\end{remark}

\section{Numerical Experiments}
\label{sec:Results}
We verify the effectiveness of our method by applying it to several existing numerical solvers and comparing its convergence speed with the original solvers. 
We also include the acceleration approaches from~\citemain{Scieur2020,fu2019anderson} for comparison.
The regularized nonlinear acceleration (RNA) proposed in~\citemain{Scieur2020} computes an accelerated iterate via an affine combination of the previous $k$ iterates, and it also introduces a quadratic regularization when computing the affine combination coefficients. Unlike our approach, it performs an acceleration step every $k$ iterations instead of every iteration, and its regularization weight is determined by a grid search that finds the weight that leads to the lowest target function value at the accelerated iterate. 
The A2DR scheme proposed in~\cite{fu2019anderson} is a globalization of $\AA$ applied on Douglas-Rachford splitting, using a quadratic regularization together  with an acceptance mechanism based on sufficient decrease of the residual.
All experiments are carried out on a laptop with a Core-i7 9750H at 2.6GHz and 16GB of RAM. 
The source code for the examples in this section is available at \url{https://github.com/bldeng/Nonmonotone-AA}.

Our method involves several parameters. The parameters $p_1$, $p_2$, $\eta_1$ and $\eta_2$, used for determining acceptance of the trial step and updating the regularization weight, are standard parameters for trust-region methods. We choose $p_1=0.01$, $p_2=0.25$, $\eta_1=2$, $\eta_2=0.25$ by default.
The parameter $\gamma$ affects the convex combination weights in computing $\res{k}$ in Eq.~\eqref{eq:ReisdualCombination}, and we choose  $\gamma = 10^{-4}$.
For the parameter $c$ in the definition of $\pred{k}$, we choose $c = \kappa$ where $\kappa < 1$ is a Lipschitz constant for the function $g$,  to satisfy the conditions for Theorems~\ref{lemma3-10} and \ref{thm:LocalConvergenceRate}. 
We will derive the value of $\kappa$ in each experiment.
The initial regularization factor $\mu_0$ is set to $\mu_0 = 1$ unless stated otherwise.
Concerning the number $m$ of previous iterates used in an $\AA$ step, we can make the following observations: a larger $m$ tends to reduce the number of iterations required for convergence, but also increases the computational cost per iteration; our experiments suggest that choosing $5 \leq m \leq 20$ often achieves a good balance.
For each experiment below, we will include multiple choices of $m$ for comparison.  
Appendix~\ref{app:ablation} provides some further ablation studies for the parameters $p_1$, $p_2$, $\eta_1$, $\eta_2$, and $c$.

\subsection{Logistic Regression}

First, to compare our method with the RNA scheme proposed by~\citemain{Scieur2020}, we consider the following logistic regression problem from~\citemain{Scieur2020} that optimizes a decision variable $x \in \mathbb{R}^n$:
\begin{equation}
	\min_{x}~F(x),
\end{equation}
where 
\begin{equation}
F(x)=\frac{1}{N}\sum\nolimits_{i=1}^N \log(1+\exp(-b_ia_i^Tx))+\frac{\tau}{2}\|x\|^2,
\label{eq:LogisticRegression}
\end{equation}
and $a_i \in \mathbb{R}^n$, $b_i \in \{-1, 1\}$ are the attributes and label of the data point $i$, respectively. 
Following~\citemain{Scieur2020}, we consider gradient descent solver $x^{k+1} = g(x^k)$ with a fixed step size:
$$g(x) = x-\frac{2}{L_F+\tau}\nabla F(x),$$
where
\begin{equation}
L_F= \tau + \frac{\|A\|_2^2}{4N}
\end{equation}
is the Lipschitz constant of $\nabla F$, and $A=[a_1,...,a_N]^T\in\mathbb{R}^{N\times n}$.
Then $g$ is Lipschitz continuous with modulus 
\[
\kappa = \frac{L_F-\tau}{L_F+\tau} < 1
\]
and differentiable, which satisfies Assumption~\ref{assum3-1}.
We apply our approach (denoted by ``LM-AA'') and RNA to this solver, and compare their performance on two datasets: \texttt{covtype}\footnote{\url{https://archive.ics.uci.edu/ml/datasets/covertype}} (54 features, 581012 points), and \texttt{sido0}\footnote{\url{http://www.causality.inf.ethz.ch/data/sido0_matlab.zip}} (4932 features, 12678 points). For each dataset, we normalize the attributes and solve the problem  with $\tau = L_F / 10^{6}$ and $\tau = L_F / 10^{9}$, respectively. 
For the implementation of RNA, we use the source code released by the authors of~\citemain{Scieur2020}\footnote{\url{https://github.com/windows7lover/RegularizedNonlinearAcceleration/tree/master/Matlab/src}}.
We set $\mu_0 = 100$, and $m = 10, 15, 20$, respectively for our method.
RNA performs an acceleration step every $k$ iterations, and we test $k=5, 10, 20$, respectively. All other RNA parameters are set to their default values as provided in the source code (in particular, with grid-search adaptive regularization weight and line search enabled). Fig.~\ref{fig:LogisticRegression} plots for each method the normalized target function value $(F(x^k) - F^*)/F^*$ 
with respect to the iteration count and computational time, where $F^\ast$ is the ground-truth global minimum computed by running each method until full convergence and taking the minimum function value among all methods. All variants of LM-AA and RNA accelerate the decrease of the target function compared with the original gradient descent solver, with our methods achieving an overall faster decrease. 
\begin{figure}[t!]
	\centering
	\includegraphics[width=\textwidth]{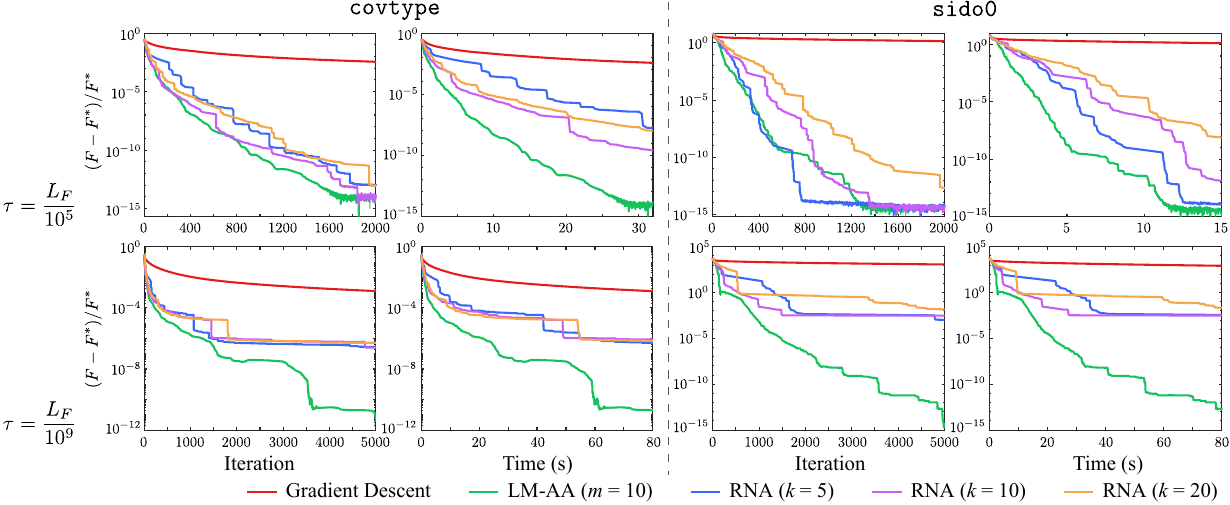}
	\caption{Comparison between RNA~\protect\cite{Scieur2020} and our method on a gradient descent solver for the logistic regression problem~\eqref{eq:LogisticRegression} for the \texttt{covtype} and \texttt{sido0} datasets, with a different choice of parameter $\tau$ in each row.}
	\label{fig:LogisticRegression}
\end{figure}

\subsection{Image Reconstruction}

Next, we consider a nonsmooth problem proposed in~\citemain{wang2008new} for total variation based image reconstruction:
\begin{equation}
	\min_{w,u} ~{\sum}_{i=1}^{N^2}\|w_i\|_2+\frac{\beta}{2}{\sum}_{i=1}^{N^2}\|w_i-D_i u\|_2^2+\frac{\fidweight}{2}\|Ku-s\|^2_2, 
	\label{prob-TV}
\end{equation}
where $s\in [0,1]^{N^2}$ is an $N \times N$ input image, $u\in\mathbb{R}^{N^2}$ is the output image to be optimized, $K\in\mathbb{R}^{N^2\times N^2}$ is a linear operator,  $D_i \in\mathbb{R}^{2 \times N^2}$ represents the discrete gradient operator at pixel $i$, $w = (w_{1}^T,\ldots,w_{N^2}^T)^T \in \R^{2N^2}$ are auxiliary variables for the image gradients, $\fidweight > 0$ is a fidelity weight, and $\beta > 0$ is a penalty parameter.
The solver in~\citemain{wang2008new} can be written as alternating minimization between $u$ and $w$ as follows:
\begin{align}
	u^{k+1} & = \argmin_{u} \frac{\beta}{2}{\sum}_{i=1}^{N^2}\|w_i^k-D_i u\|_2^2+\frac{\fidweight}{2}\|Ku-s\|^2_2, \label{eq:UUpdate}\\
	w^{k+1} &= \argmin_{w} {\sum}_{i=1}^{N^2}\|w_i\|_2+\frac{\beta}{2}{\sum}_{i=1}^{N^2}\|w_i-D_i u^{k+1}\|_2^2 \label{eq:WUpdate}.
\end{align}
The solutions to the subproblems \eqref{eq:UUpdate} and \eqref{eq:WUpdate} can both be computed in a closed form.
When $\beta$ and $\fidweight$ are fixed, this can be treated as a fixed-point iteration $w^{k+1} = g(w^k)$, and it satisfies Assumption~\ref{assum3-1} (see Appendix~\ref{sec:app-tv} for a detailed derivation of $g$ and verification of Assumption~\ref{assum3-1}).
In the following, we consider the solver with $K = I$ and $\fidweight = 4$ for image denoising.
In this case, condition~\ref{Bone} is satisfied with 
\[
\kappa = 1-\left(1+\frac{4\beta}{\fidweight}\right)^{-1}
\] 
(see Appendix~\ref{sec:app-tv} for the derivation).
We apply this solver to a $1024 \times 1024$ image with added Gaussian noise that has a zero mean and a variance of $\sigma = 0.05$  (see Fig.~\ref{fig:TVDenoising}).
\begin{figure}[t!]
	\centering
	\includegraphics[width=\textwidth]{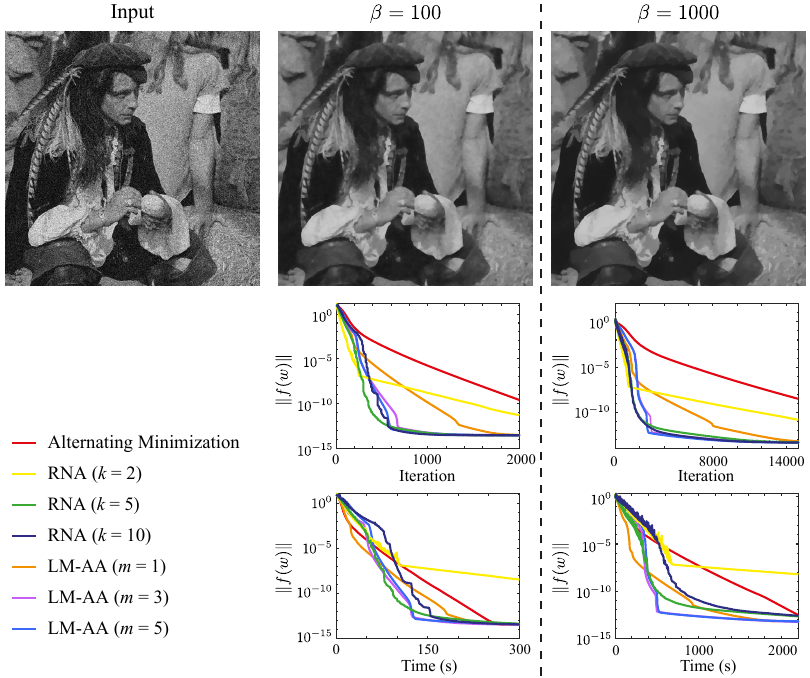}
	\caption{Application of our method to an alternating minimization solver for an image denoising problem~\eqref{prob-TV}.}
	\label{fig:TVDenoising}
\end{figure}
We use the source code released by the authors of~\citemain{wang2008new}\footnote{\url{https://www.caam.rice.edu/~optimization/L1/ftvd/v4.1/}} for the implementation of this solver, and apply our acceleration method with $m = 1, 3, 5$ respectively.
For comparison, we also apply the RNA scheme with $k = 2, 5, 10$, respectively. Here we choose smaller values of $m$ and $k$ than the logistic regression example because both RNA and our method have a high relative overhead on this problem, which means that larger values of $m$ or $k$ may induce overhead that offsets the performance gain from acceleration.
Similar to the logistic regression example, we use the released source code of RNA for our experiments, and set all RNA parameters to their default values. 
Fig.~\ref{fig:TVDenoising} plots the residual norm $\|f(w)\|$ for all methods, with $\beta = 100$ and $\beta=1000$, respectively. 
All instances of acceleration methods converge faster to a fixed point than the original alternating minimization solver, except for RNA with $k=2$ which is slower in terms of the actual computational time due to its overhead for the grid search of regularization parameters. Overall, the two acceleration approaches achieve a rather similar performance on this problem.

\subsection{Nonnegative Least Squares}

Finally, to compare our method with~\citemain{fu2019anderson}, we consider a nonnegative least squares (NNLS) problem that is used in~\citemain{fu2019anderson} for evaluation: 
\begin{equation}
\label{eq:reNNLS}
\min_{x}~\psi(x)+\varphi(x),
\end{equation}
where $x=(x_1,x_2)\in\mathbb{R}^{2q}$, $\psi(x)=\|Hx_1-t\|_2^2+\mathcal{I}_{x_2\geq 0}(x_2)$, and $\varphi(x)=\mathcal{I}_{x_1=x_2}(x)$, 
with $\mathcal{I}_S$ being the indicator function of the set $S$. The Douglas-Rachford splitting (DRS) solver for this problem can be written as
\begin{equation} v^{k+1} =g(v^k) ={\textstyle{\frac{1}{2}}}((2\prox_{\beta\varphi}-I)(2\prox_{\beta \psi}-I)+I)v^k
\label{eq:NNLSDRS}
\end{equation}
where $v^{k}=(v_1^k,v_2^k)\in\mathbb{R}^{2q}$ is an auxiliary variable for DRS and $\beta$ is the penalty parameter. 
In~\citemain{fu2019anderson}, the authors use their regularized $\AA$ method (A2DR) to accelerate the DRS solver~\eqref{eq:NNLSDRS}. 
To apply our method, we verify in Appendix~\ref{appx:NNLS} that if $H$ is of full column rank, then $g$ satisfies condition~\ref{Bone} with 
$$\kappa = \frac{\sqrt{3+c_1^2}}{2} < 1$$ where $c_1=\max\{\frac{\beta\sigma_1-1}{\beta\sigma_1+1},\frac{1-\beta\sigma_0}{1+\beta\sigma_0}\}$, and $\sigma_0,\sigma_1$ are the minimal and maximal eigenvalues of $2H^T H$, respectively. Moreover, $g$ is also differentiable under a mild condition.
We compare our method with A2DR on the solver~\eqref{eq:NNLSDRS}, with the same $\AA$ parameters $m = 10, 15, 20$.
The methods are tested using a $600 \times 300$ sparse random matrix $H$ with $1\%$ nonzero entries and a random vector $t$. We use the source code released by the authors of~\citemain{fu2019anderson}\footnote{\url{https://github.com/cvxgrp/a2dr}} for the implementation of A2DR, and set all A2DR parameters to their default values. While A2DR and DRS are implemented using parallel evaluation of the proximity operators in the released A2DR code, we implement our method as a single-threaded application for simplicity. 
Fig~\ref{fig:NNLS} plots the residual norm $\|f(v)\|$ for DRS and the two acceleration methods. It also plots the norm of the overall residual $r = (r_{\mathrm{prim}}, r_{\mathrm{dual}})$ used in~\citemain{fu2019anderson} for measuring convergence, where $r_{\mathrm{prim}}$ and $r_{\mathrm{dual}}$ denote the primal and dual residuals as defined in Equations~(7) and (8) of~\citemain{fu2019anderson}, respectively. For both residual measures, the original DRS solver converges slowly after the initial iterations, whereas the two acceleration methods achieve significant speedup.
Moreover, the single-threaded implementation of our method outperforms the parallel A2DR with the same $m$ parameter, in terms of both iteration count and computational time.
\begin{figure}[t!]
	\centering
	\includegraphics[width=0.95\columnwidth]{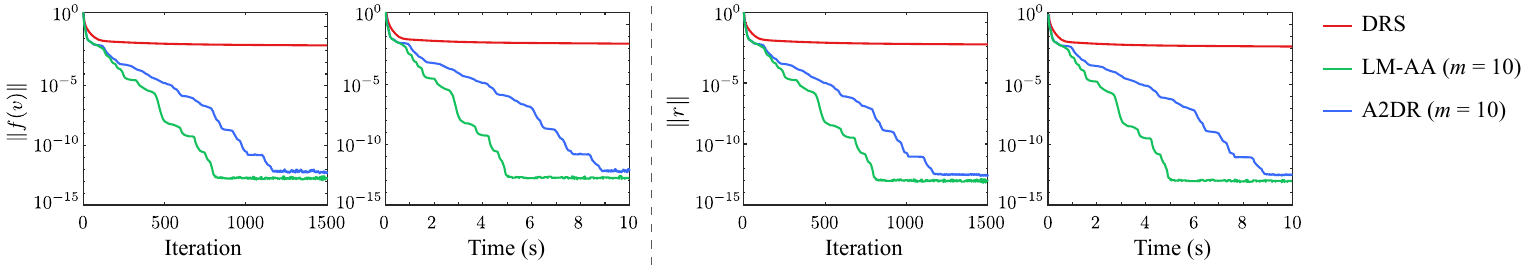}
	\caption{Comparison between A2DR~\protect\citemain{fu2019anderson} and our method on the NNLS solver~\eqref{eq:NNLSDRS} with an $600 \times 300$ sparse random matrix $H$ and a random vector $t$. }
	\label{fig:NNLS}
\end{figure}

\subsection{Statistics of Successful Steps}
Our acceptance mechanism plays a key role in achieving global and local convergence of the proposed method. 
To demonstrate its behavior, we provide statistics of the successful steps in Figs.~\ref{fig:LogisticRegression},  \ref{fig:TVDenoising} and \ref{fig:NNLS}.  Specifically, for each instance of LM-AA, we count the total steps required to reach a certain level of accuracy and we compare it with the corresponding number of successful $\AA$ steps within these steps.
Tables~\ref{tab:01}, \ref{tab:02}, \ref{tab:03}, and \ref{tab:04} show the statistics of successful steps for Figs.~\ref{fig:LogisticRegression},  \ref{fig:TVDenoising} and \ref{fig:NNLS}, respectively. Here, besides the total number of steps, we report the success rate which is defined as the ratio between successful and total steps required to reach different levels of accuracy.

\begin{table*}[t!]
	\caption{Statistics of successful steps of LM-AA for the logistic regression problem \eqref{eq:LogisticRegression} and the dataset \texttt{covtype}. In each of the columns \textit{iter}, we report the number of iterations required to satisfy $(F(x^k)-F^*)/F^* \leq \textit{tol}$ with $\textit{tol} \in \{10^{-3},10^{-6},10^{-9},10^{-12},10^{-15}\}$. The columns \textit{$s$-rate} show the corresponding success rate of $\AA$-steps, i.e., \textit{$s$-rate} is  the ratio between successful and total steps required to reach the different accuracies.}
	\label{tab:01}
	\centering
	\setlength{\tabcolsep}{3pt}
	\begin{tabular}{|c||c|c|cl|cl|cl|cl|cl|}
		\hline
		\multirow{9}{*}{\STAB{\rotatebox[origin=c]{90}{\texttt{covtype}}}} &  & \multirow{3}{*}{$m$} & \multicolumn{10}{c|}{{\textit{tol} \quad (stopping criterion: $(F-F^*)/F^* \leq \textit{tol}$)}} \\ 
		\cline{4-13}
		\rule{0pt}{10pt} & & & \multicolumn{2}{c|}{$10^{-3}$} & \multicolumn{2}{c|}{$10^{-6}$} & \multicolumn{2}{c|}{$10^{-9}$} & \multicolumn{2}{c|}{$10^{-12}$} & \multicolumn{2}{c|}{$10^{-15}$} \\
		& & & \textit{$s$-rate} & \textit{iter} & \textit{$s$-rate} & \textit{iter} & \textit{$s$-rate} & \textit{iter} & \textit{$s$-rate} & \textit{iter}  & \textit{$s$-rate} & \textit{iter} \\
		\cline{2-13}
		\rule{0pt}{8pt} & \multirow{3}{*}{$\tau = \frac{L_F}{10^6}$} & 10 & 72.3\% & 83 & 94.3\% & 599 & 98.3\% & 2000 & 98.3\% & 2000 & 98.3\% & 2000 \\
		\rule{0pt}{6pt} & & 15 & 57.3\% & 82 & 88.0\% & 541 & 93.8\% & 1175 & 96.1\% & 1887 & 96.4\% & 2000 \\
		\rule{0pt}{6pt} & & 20 & 49.0\% & 98 & 73.5\% & 434 & 87.8\% & 1058 & 90.5\% & 1351 & 92.9\% & 1812 \\
		\cline{2-13}
		\rule{0pt}{8pt} & \multirow{3}{*}{$\tau = \frac{L_F}{10^9}$} & 10 & 86.6\% & 119 & 95.7\% & 1083 & 96.3\% & 3336 & 96.7\% & 3755 & 96.7\% & 3781 \\
		\rule{0pt}{6pt} & & 15 & 55.4\% & 83 & 86.0\% & 783 & 32.0\% & 5000 & 32.0\% & 5000 & 32.0\% & 5000  \\
		\rule{0pt}{6pt} & & 20 & 50.5\% & 91 & 73.2\% & 593 & 18.8\% & 5000 & 18.8\% & 5000 & 18.8\% & 5000 \\
		\hline
	\end{tabular}
\end{table*}

\begin{table*}[t!]
	\caption{Statistics of successful steps of LM-AA for the logistic regression problem \eqref{eq:LogisticRegression} and the dataset \texttt{sido0}. In each of the columns \textit{iter}, we report the number of iterations required to satisfy the criterion $(F(x^k)-F^*)/F^* \leq \textit{tol}$ with $\textit{tol} \in \{10^{-3},10^{-6},10^{-9},10^{-12},10^{-15}\}$. The columns \textit{$s$-rate} show the corresponding success rate of $\AA$-steps, i.e., \textit{$s$-rate} is  the ratio between successful and total steps required to reach the different accuracies.}
	\label{tab:02}
	\centering
	\setlength{\tabcolsep}{3pt}
	\begin{tabular}{|c||c|c|cl|cl|cl|cl|cl|}
		\hline
		\multirow{9}{*}{\STAB{\rotatebox[origin=c]{90}{\texttt{sido0}}}} &  & \multirow{3}{*}{$m$} & \multicolumn{10}{c|}{{\textit{tol} \quad (stopping criterion: $(F-F^*)/F^* \leq \textit{tol}$)}} \\ 
		\cline{4-13}
		\rule{0pt}{10pt} & & & \multicolumn{2}{c|}{$10^{-3}$} & \multicolumn{2}{c|}{$10^{-6}$} & \multicolumn{2}{c|}{$10^{-9}$} & \multicolumn{2}{c|}{$10^{-12}$} & \multicolumn{2}{c|}{$10^{-15}$} \\
		& & & \textit{$s$-rate} & \textit{iter} & \textit{$s$-rate} & \textit{iter} & \textit{$s$-rate} & \textit{iter} & \textit{$s$-rate} & \textit{iter}  & \textit{$s$-rate} & \textit{iter} \\
		\cline{2-13}
		\rule{0pt}{8pt} & \multirow{3}{*}{$\tau = \frac{L_F}{10^6}$} & 10 & 67.1\% & 334 & 82.9\% & 686 & 88.3\% & 997 & 92.7\% & 1621 & 93.8\% & 2000 \\
		\rule{0pt}{6pt} & & 15 & 55.6\% & 315 & 77.7\% & 676 & 84.8\% & 994 & 89.7\% & 1481 & 91.8\% & 1992 \\
		\rule{0pt}{6pt} & & 20 & 46.2\% & 368 & 65.4\% & 619 & 76.7\% & 917 & 85.3\% & 1453 & 88.8\% & 2000 \\
		\cline{2-13}
		\rule{0pt}{8pt} & \multirow{3}{*}{$\tau = \frac{L_F}{10^9}$} & 10 & 48.3\% & 1208 & 64.6\% & 2051 & 67.1\% & 3006 & 53.5\% & 4774 & 51.9\% & 4950 \\
		\rule{0pt}{6pt} & & 15 & 46.5\% & 1815 & 59.5\% & 2714 & 60.5\% & 3788 & 50.9\% & 5000 & 50.9\% & 5000 \\
		\rule{0pt}{6pt} & & 20 & 46.6\% & 2205 & 57.8\% & 3037 & 57.8\% & 4090 & 51.7\% & 5000 & 51.7\% & 5000 \\
		\hline
	\end{tabular}
\end{table*}

\begin{table*}[t!]
	\caption{Statistics of successful steps of LM-AA for the image denoising problem \eqref{prob-TV}. In each of the columns \textit{iter}, we report the number of iterations required to satisfy the criterion $\|f(w^k)\| \leq \textit{tol}$ with $\textit{tol} \in \{10^{-3},10^{-6},10^{-9},10^{-12},10^{-15}\}$. The columns \textit{$s$-rate} show the success rate of $\AA$-steps, i.e., \textit{$s$-rate} is  the ratio between successful and total steps required to reach the different accuracies.}
	\label{tab:03}
	\centering
	\setlength{\tabcolsep}{3pt}
	\begin{tabular}{|c|c|cl|cl|cl|cl|cl|}
		\hline
		& \multirow{3}{*}{$m$} & \multicolumn{10}{c|}{{\textit{tol} \quad (stopping criterion: $\|f(w)\| \leq \textit{tol}$)}} \\ 
		\cline{3-12}
		\rule{0pt}{10pt} & & \multicolumn{2}{c|}{$10^{-3}$} & \multicolumn{2}{c|}{$10^{-6}$} & \multicolumn{2}{c|}{$10^{-9}$} & \multicolumn{2}{c|}{$10^{-12}$} & \multicolumn{2}{c|}{$10^{-15}$} \\
		& & \textit{$s$-rate} & \textit{iter} & \textit{$s$-rate} & \textit{iter} & \textit{$s$-rate} & \textit{iter} & \textit{$s$-rate} & \textit{iter}  & \textit{$s$-rate} & \textit{iter} \\
		\cline{1-12}
		\rule{0pt}{8pt} \multirow{3}{*}{$\beta = 100$} & 1 & 93.7\% & 190 & 96.6\% & 378 & 98.4\% & 812 & 99.0\% & 1274 & 91.2\% & 2000 \\
		\rule{0pt}{6pt} & 3 & 39.5\% & 223 & 54.4\% & 296 & 72.0\% & 483 & 79.7\% & 666 & 64.3\% & 2000 \\
		\rule{0pt}{6pt} & 5 & 38.8\% & 227 & 53.4\% & 298 & 68.8\% & 446 & 75.9\% & 577 & 54.8\% & 2000 \\
		\cline{1-12}
		\rule{0pt}{8pt} \multirow{3}{*}{$\beta = 1000$} & 1 & 92.1\% & 1057 & 90.5\% & 1489 & 95.9\% & 3477 & 97.9\% & 7677 & 75.7\% & 15000 \\
		\rule{0pt}{6pt} & 3 & 34.6\% & 1458 & 39.6\% & 1824 & 56.7\% & 2554 & 68.8\% & 3543 & 47.6\% & 15000 \\
		\rule{0pt}{6pt} & 5 & 35.0\% & 1410 & 38.5\% & 1776 & 60.6\% & 2773 & 68.0\% & 3408 & 43.0\% & 15000 \\
		\hline
	\end{tabular}
\end{table*}

\begin{table*}[t!]
	\caption{Statistics of successful steps of LM-AA for the nonnegative least squares problem \eqref{eq:reNNLS}. In each of the columns \textit{iter}, we report the number of iterations required to satisfy the criterion $\|f(w^k)\| \leq \textit{tol}$ with $\textit{tol} \in \{10^{-3},10^{-6},10^{-9},10^{-12},10^{-15}\}$. The columns \textit{$s$-rate} show the success rate of $\AA$-steps, i.e., \textit{$s$-rate} is the ratio between successful and total steps required to reach the different accuracies.}
	\label{tab:04}
	\centering
	\setlength{\tabcolsep}{3pt}
	\begin{tabular}{|c|cl|cl|cl|cl|cl|}
		\hline
		\multirow{3}{*}{$m$} & \multicolumn{10}{c|}{{\textit{tol} \quad (stopping criterion: $\|f(v)\| \leq \textit{tol}$)}} \\ 
		\cline{2-11}
		\rule{0pt}{10pt} & \multicolumn{2}{c|}{$10^{-3}$} & \multicolumn{2}{c|}{$10^{-6}$} & \multicolumn{2}{c|}{$10^{-9}$} & \multicolumn{2}{c|}{$10^{-12}$} & \multicolumn{2}{c|}{$10^{-15}$} \\
		& \textit{$s$-rate} & \textit{iter} & \textit{$s$-rate} & \textit{iter} & \textit{$s$-rate} & \textit{iter} & \textit{$s$-rate} & \textit{iter}  & \textit{$s$-rate} & \textit{iter} \\
		\cline{1-11}
		\rule{0pt}{8pt} 5 & 100\% & 646 & 100\% & 1500 & 100\% & 1500 & 100\% & 1500 & 100\% & 1500 \\
		\rule{0pt}{6pt} 10 & 100\% & 235 & 100\% & 512 & 100\% & 696 & 93.3\% & 989 & 67.5\% & 1500 \\
		\rule{0pt}{6pt} 15 & 100\% & 151 & 100\% & 437 & 100\% & 680 & 100\% & 792 & 58.0\% & 1500 \\
		\rule{0pt}{6pt} 20 & 100\% & 162 & 100\% & 344 & 100\% & 465 & 100\% & 572 & 44.4\% & 1500 \\
		\hline
	\end{tabular}
\end{table*}

The results in Table~\ref{tab:04} demonstrate that essentially all $\AA$ steps are accepted in the nonnegative least squares problem. This observation is also independent of the choice of the parameter $m$. More specifically, the success rate of $\AA$ steps only decreases and more alternative fixed-point iterations are performed when we seek to solve the problem with the highest accuracy $\textit{tol} = 10^{-15}$. Table~\ref{tab:03} illustrates that a similar behavior can also be observed for the image denoising problem \eqref{prob-TV} when setting $m=1$. Notice that this high accuracy is close to machine precision and hence this effect is mainly caused by numerical errors and inaccuracies that affect the computation and quality of an $\AA$ step. The results in Table~\ref{tab:03} also demonstrate a second typical effect: the success rate of $\AA$ step is often lower when the chosen accuracy is relatively low. With increasing accuracy, the rate then increases to around 70\%--80\%. This general observation is also supported by our results for logistic regression, see Tables~\ref{tab:01} and \ref{tab:02}. (Here the maximum success rate is more sensitive to the choice of $m$, $\tau$, and of the dataset). 

In summary, the statistics provided in Tables~\ref{tab:01}, \ref{tab:02}, \ref{tab:03}, and \ref{tab:04} support our theoretical results. The success rate of $\AA$ steps gradually increases as the iteration gets closer to the fixed point, which indicates a transition to a pure regularized $\AA$ scheme. Furthermore, as more $\AA$ steps seem to be rejected at the beginning of the iterative procedure, our globalization mechanism effectively guarantees global progress and convergence of the approach. 


\section{Conclusions}
We propose a novel globalization technique for Anderson acceleration which combines adaptive quadratic regularization and a nonmonotone acceptance strategy. We prove the global convergence of our approach under mild assumptions. Furthermore, we show that the proposed globalized $\AA$ scheme has the same local convergence rate as the original $\AA$ iteration and that the globalization mechanism does not hinder the acceleration effect of $\AA$. This is one of the first $\AA$ globalization methods that achieves global convergence and fast local convergence simultaneously. Several numerical examples illustrate that our method is competitive and it can improve the efficiency of a variety of numerical solvers.

%
%

{
\bibliographystyle{spmpsci}      
\bibliography{Commonbib}   
}

\section*{Statements and Declarations}

\paragraph{Funding} 
A.~Milzarek was partly supported by the Fundamental Research Fund -- Shenzhen Research Institute for Big Data (SRIBD) Startup Fund JCYJ-AM20190601.
B.~Deng was partly supported by the Guangdong International Science and Technology Cooperation Project (No. 2021A0505030009).

\paragraph{Competing Interests}
The authors have no relevant financial or non-financial interests to disclose.

\paragraph{Data Availability}
The datasets generated during and/or analysed during the current study are available in the GitHub repository \url{https://github.com/bldeng/Nonmonotone-AA}. 

\appendix
\section{Proof that $\pred{k}$ in Eq.~\eqref{eq:ReisdualCombination} is Positive}
\label{proof:PositivePredk}
\begin{proof}
By the definition of the minimization problem~\eqref{eq:RegularizedOpt}, we have
\begin{align*}
	\|\hat{f}^k(\alpha^k)\|^2 & \leq \|\hat{f}^k(\alpha^k)\|^2 + \lambda_k \|\alpha^k\|^2 \leq \|\hat{f}^k(0)\|^2 +  \lambda_k \|0\|^2 = \|f^{k_0}\|^2.
\end{align*}
Then Eqs.~\eqref{eq:ReisdualCombination}, \eqref{eq:PermutationRule} and $c \in (0,1)$ imply that 
\begin{equation} \label{eq:esti-app-res}
r_k \geq \|f^{k_0}\| \geq \|\hat{f}^k(\alpha^k)\| \geq  c \|\hat{f}^k(\alpha^k)\|.
\end{equation}
By the algorithmic construction we know $\|f^{k_0}\|>0$. So if $\|\hat{f}^k(\alpha^k)\|=0$ then the second inequality is strict, otherwise the third inequality is strict. Overall we can deduce that $\pred{k}= r_k - c \|\hat{f}^k(\alpha^k)\|$ must be positive.
\end{proof}

\section{Verification of Local Convergence Assumptions}
\label{app:Verification}

In this section, we briefly discuss different situations that allow us to verify and establish the local conditions stated in Assumption~\ref{assum3-1} and required for Corollary~\ref{cor:stronger-loc}. 

\subsection{The Smooth Case} \label{sec:app-smo}

Clearly, assumption \ref{Btwo} is satisfied if $g$ is a smooth mapping. 

In addition, as mentioned at the end of subsection~\ref{sec:loc}, if the mapping $g$ is continuously differentiable in a neighborhood of its associated fixed-point $x^*$, then the stronger differentiability condition \vspace{.5ex}

\begin{enumerate}[label=\textup{\textrm{(C.\arabic*)}},topsep=0pt,itemsep=0.5ex,partopsep=0ex,parsep=0ex,leftmargin=8ex]
\item \label{Cone} $\|g(x)-g(x^*) - g^\prime(x^*)(x-x^*)\| = O(\|x-x^*\|^2)$ for $x \to x^*$, \vspace{.5ex}
\end{enumerate}
used in Corollary~\ref{cor:stronger-loc}, holds if the derivative $g^\prime$ is locally Lipschitz continuous around $x^*$, i.e., for any $x,y \in \mathbb{B}_\epsilon(x^*)$ we have
\begin{equation} \label{eq:lip-fre} \|g^\prime(x) - g^\prime(y)\| \leq L \|x-y\|. \end{equation}
Assumption \ref{Cone} can then be shown via utilizing a Taylor expansion. Let us notice that for \ref{Cone} it is enough to fix $y = x^*$ in \eqref{eq:lip-fre}. Such a condition is known as outer Lipschitz continuity at $x^*$. Furthermore, assumption \ref{Bone} holds if $\sup_{x \in \Rn} \|g^\prime(x)\| < 1$, see, e.g., Theorem 4.20 of~\citeappx{Bec14}.

\subsection{Total Variation Based Image Reconstruction} \label{sec:app-tv}

The alternating minimization solver for image reconstruction problem~\eqref{prob-TV} can be written as a fixed-point iteration 
\begin{equation} \label{eq:tv-fix} 
w^{k+1} = g(w^k) := (\Phi\circ h)(w^k),
\end{equation}
with 
\begin{align*}
\Phi(w) &:= (s_\beta(w_{1})^T,\ldots,s_\beta(w_{N^2})^T)^T \in \mathbb{R}^{2N^2}, \quad s_\beta(x) :=\max\left\{\|x\|-\frac{1}{\beta},0\right\}\frac{x}{\|x\|},
\end{align*}
and $h(w) := DM^{-1}(D^T w+ ({\nu}/{\beta})K^Tf$, $D := (D_1^T,\ldots,D_{N^2}^T)^T$, and $M :=D^TD+({\nu}/{\beta})K^TK$.
Let us notice that the mapping $h$ and the fixed point iteration \eqref{eq:tv-fix} are well-defined when the null spaces of $K$ and $D$ have no intersection, see, e.g., Assumption~1 of~\citeappx{wang2008new}. 

Next, we verify Assumption~\ref{assum3-1} for the mapping $g$ in \eqref{eq:tv-fix}.
\begin{proposition}
\label{prop4-2}
Suppose that the operator $K^TK$ is invertible. Then, the spectral radius $\rho(T)$ of $T := DM^{-1}D^T$ fulfills $\rho(T)<1$ and condition \ref{Bone} is satisfied.
\end{proposition}
\begin{proof}
Utilizing the Sherman-Morrison-Woodbury formula and the invertibility of $K^T K$, we obtain
\begin{align}  
	(I+\xi D(K^TK)^{-1}D^T)^{-1} & = I - \xi D(K^TK)^{-1}(I + \xi D^TD(K^TK)^{-1})^{-1} D^T = I-T, 
	\label{eq:smw} 
\end{align}
where $\xi := {\beta}/{\nu}$. Due to $\lambda_{\min}(I+\xi D(K^TK)^{-1}D^T) \geq 1$ and 
$\lambda_{\max}(I+\xi D(K^TK)^{-1}D^T) \leq 1 + \xi \|D(K^TK)^{-1}D^T\|,$
it then follows that
\begin{align*}
\sigma((I+\xi D(K^TK)^{-1}D^T)^{-1}) \subset \left [\frac{1}{1+\xi\|D(K^TK)^{-1}D^T\|},1 \right],
\end{align*}
where $\sigma(\cdot)$ denotes the spectral set of a matrix. Combining this observation with \eqref{eq:smw}, we have
\[ \sigma(T)\subset \left[0,1-\frac{1}{1+\xi\|D(K^TK)^{-1}D^T\|}\right] \quad \implies \quad \rho(T)<1. \]
Furthermore, following the proof of Theorem 3.6 in~\citeappx{wang2008new}, it holds that
$$ \|g(w)-g(v)\|  \leq \rho(T)\|w-v\| \quad \forall~w,v \in\mathbb{R}^{2N^2}  $$
and hence, assumption \ref{Bone} is satisfied. \qed
\end{proof}

Concerning assumption \ref{Btwo}, it can be shown that $g$ is twice continuously differentiable on the set 
$$\mathcal W =\{w: \|[h(w)]_{i}\|\neq {1}/{\beta}, \,\forall~i=1,\ldots,N^2\}.$$ 
(In this case the max-operation in the shrinkage operator $s_\beta$ is not active). 
Moreover, since $h$ is continuous, the set $\mathcal W$ is open. Consequently, for every point $w \in \mathcal W$, we can find a bounded open neighborhood $N(w)$ of $w$ such that $N(w) \subset \mathcal W$. Hence, if the mapping $g$ has a fixed-point $w^*$ satisfying $w^* \in \mathcal W$, then we can infer that $g$ is differentiable on $N(w^*)$ and assumption \ref{Btwo} has to hold at $w^*$. Furthermore, the stronger assumption \ref{Cone} for Corollary~\ref{cor:stronger-loc} is satisfied as well in this case. Finally, if $K$ is the identity matrix, then notice that the finite difference matrix $D$ satisfies that $\|D\|\leq 2$ and we can infer $\rho(T)\leq 1-(1+4\beta/\nu)^{-1}$ which 
justifies the choice of $c$ in our algorithm.

\subsection{Nonnegative Least Squares}
\label{appx:NNLS}
We first note that given the specific form of $\varphi$, we can calculate the proximity operator $\varphi$ explicitly as 
\[
\prox_{\beta\varphi}(v)=\frac12((v_1+v_2)^T,(v_1+v_2)^T)^T,
\]
where $v=(v_1^T,v_2^T)^T$. Consequently, we obtain $[2\prox_{\beta\varphi}-I](v)=(v_2^T,v_1^T)^T.$ Similarly, by setting 
$ \psi_1(v_1) := \|Hv_1 - t\|_2^2$ and $\psi_2(v_2) := \mathcal I_{v_2\geq0}(v_2),$
we have
\begin{align*} 
\prox_{\beta \psi}(v) & = 
\begin{pmatrix} 
	\prox_{\beta \psi_1}(v_1) \\
	\prox_{\beta \psi_2}(v_2) 
\end{pmatrix},\\
\prox_{\beta \psi_1}(v_1) &= 
(H^T H + (2\beta)^{-1}I )^{-1}(H^T t + v_1 /({2\beta})),\\
\prox_{\beta \psi_2}(v_2) &= \mathcal P_{[0,\infty)^q}(v_2),
\end{align*}
where $\mathcal P_{[0,\infty)^q}$ denotes the Euclidean projection onto the set of nonnegative numbers $[0,\infty)^q$. In the next proposition, we give a condition to establish \ref{Bone} for $g$. 
\begin{proposition}
Let $\sigma_0$ and $\sigma_1$ denote the minimum and maximum eigenvalue of $2H^TH$, respectively and suppose $\sigma_0 > 0$. Then, the mapping $g$ is Lipschitz continuous with modulus ${\sqrt{3+c_1^2}}/{2}$, where $c_1=\max\{\frac{\beta\sigma_1-1}{\beta\sigma_1+1},\frac{1-\beta\sigma_0}{1+\beta\sigma_0}\} < 1$.
\end{proposition}
\begin{proof}
We can explicitly calculate $g$ as follows 
\begin{equation} 
	g(v) = \frac{1}{2}((2\prox_{\beta\varphi}-I)(2\prox_{\beta \psi}-I)+I)v = \frac12 \begin{pmatrix}\mathcal R^2_\beta(v_2)+v_1 \\ \mathcal R^1_\beta(v_1)+v_2\end{pmatrix},
	\label{eq:dr-nn}  
\end{equation}
where $\mathcal R^i_\beta=2\prox_{\beta \psi_i}-I$, $i=1,2$. By Proposition~4.2 of~\citeappx{BauCom11}, the reflected operators $\mathcal R_\beta^1$ and $\mathcal{R}_\beta^2$ are nonexpansive. Moreover, since $\sigma_0>0$, $\psi_1$ is strongly convex with modulus $\sigma_0$ and $\sigma_1$-smooth. Then by Theorem~1 of~\citeappx{giselsson2016linear}, $\mathcal{R}_\beta^1$ is Lipschitz continuous with modulus $c_1$.
Next, for any $v,\bar{v}\in \R^{2q}$, we have
\begingroup
\allowdisplaybreaks
\begin{align*}
\notag \|g(v)-g(\bar v)\| & \\ \notag & \hspace{-12ex} = \frac12 \left[ \|\mathcal R_\beta^2(v_2)-\mathcal R_\beta^2(\bar v_2)  +v_1-\bar v_1\|^2 + \|\mathcal R_\beta^1(v_1)-\mathcal R_\beta^1(\bar v_1)  +v_2-\bar v_2\|^2 \right]^\frac12 \\
\notag & \hspace{-12ex} \leq \frac12 \left[ (c_1^2 + 1)\|v_1 - \bar v_1\|^2 + 2(c_1 + 1)\|v_1 - \bar v_1\|  \|v_2-\bar v_2\|  + 2 \|v_2 - \bar v_2\|^2 \right]^\frac12 \\
& \hspace{-12ex} \leq  \frac12 \left[ (c_1^2 + 1)\|v_1 - \bar v_1\|^2 +\frac{(c_1+1)^2}{c_1^2+1}\|v_1 - \bar v_1\|^2+(c_1^2+1)\|v_2-\bar v_2\|^2 + 2 \|v_2 - \bar v_2\|^2 \right]^\frac12 \\
\notag & \hspace{-12ex} \leq  \frac12 \sqrt{(c_1^2 + 3)\|v_1 - \bar v_1\|^2 + (c_1^2+3) \|v_2 - \bar v_2\|^2} \leq \frac{\sqrt{3+c_1^2}}{2}\|v-\bar{v}\|,
\end{align*}
\endgroup
where we used Cauchy's inequality, the nonexpansiveness of $\mathcal R^2_\beta$, and the Lipschitz continuity of $\mathcal R^1_\beta$. The estimate in the second to last line follows from Young's inequality. \qed
\end{proof}
Hence, assumption~\ref{Bone} is satisfied if $H$ has full column rank. 
Using the special form of the mapping $g$ in \eqref{eq:dr-nn}, we see that $g$ is twice continuously differentiable at $v = (v_1^T,v_2^T)^T$ if and only if $v \in \mathcal V := \R^q \times \prod_{i=1}^q \R \backslash \{0\}$, i.e., if none of the components of $v_2$ are zero. As before, we can then infer that assumption~\ref{Btwo} and the stronger condition \ref{Cone} have to hold at every fixed-point $v^*$ of $g$ satisfying $v^* \in \mathcal V$. 
\subsection{Further Extensions}
\label{sec:app-extension}
We now formulate a possible extension of the conditions presented in Section~\ref{sec:app-smo} to the nonsmooth setting. 

If the mapping $g$ has more structure and is connected to an underlying optimization problem like in forward-backward and Douglas-Rachford splitting, nonsmoothness of $g$ typically results from the proximity operator or projection operators. In such a case, further theoretical tools are available and for certain function classes it is possible to fully characterize the differentiability of $g$ at $x^*$ via a so-called strict complementarity condition. In fact, the conditions $w^* \in \mathcal W$ and $v^* \in \mathcal V$ from section~\ref{sec:app-tv} and \ref{appx:NNLS} are equivalent to such a strict complementarity condition. In the case of forward-backward splitting, a related and in-depth discussion of this important observation is provided in \citeappx{SteThePat17,mai2019anderson} and we refer the interested reader to~\citeappx{PolRoc96,milzarek2016numerical,LiaFadPey17,SteThePat17} for further background.

Concerning the stronger assumption \ref{Cone}, we can establish the following characterization: Suppose that $g$ is locally Lipschitz continuous and let us consider the properties:
\begin{itemize}
\item The function $g$ is differentiable at $x^*$. 
\item The mapping $g$ is strongly (or 1-order) semismooth at $x^*$~\citeappx{QiSun93}, i.e., we have 
\[ \sup_{M \in \partial g(x)}~\|g(x) - g(x^*) - M(x-x^*) \| = O(\|x-x^*\|^2) \]
when $x \to x^*$. Here, the multifunction $\partial g : \Rn \rightrightarrows \R^{n \times n}$ denotes Clarke's subdifferential of the locally Lipschitz continuous (and possibly nonsmooth) function $g$, see, e.g., \citeappx{Cla90,QiSun93}. 
\item There exists an outer Lipschitz continuous selection $M^* : \R^n \to \R^{n \times n}$ of $\partial g$ in a neighborhood of $x^*$, i.e., for all $x$ sufficiently close to $x^*$ we have $M^*(x) \in \partial g(x)$ and 
\[\quad \|M^*(x) - M^*(x^*)\| \leq L_M \|x - x^*\| \]
for some constant $L_M > 0$.
\end{itemize}
Then, the mapping $g$ satisfies the condition \ref{Cone} at $x^*$. 

\begin{proof} The combination of differentiability and semismoothness implies that $g$ is strictly differentiable at $x^*$ and as a consequence, Clarke's subdifferential at $x^*$ reduces to the singleton $\partial g(x^*) = \{g^\prime(x^*)\}$. We refer to~\citeappx{QiSun93,rockafellar2009variational,milzarek2016numerical} for further details. Thus, we can infer $M^*(x^*) = g^\prime(x^*)$ and we obtain
\begin{align*} 
	\|g(x) - g(x^*) - g^\prime(x^*)(x-x^*)\| & \leq \|g(x)-g(x^*)-M^*(x)(x-x^*)\| \\
	& \hspace{4ex} + \|[M^*(x)-M^*(x^*)](x-x^*)\|\\ 
	& \leq O(\|x-x^*\|^2) + L_M \|x-x^*\|^2, 
\end{align*}
for $x \to x^*$, where we used the strong semismoothness and outer Lipschitz continuity in the last step. This establishes \ref{Cone}. \qed \end{proof}

The class of strongly semismooth functions is rather rich and includes, e.g., piecewise twice continuously differentiable (PC${}^2$) functions~\citeappx{Ulb11}, eigenvalue and singular value functions~\citeappx{SunSun02,SunSun05}, and certain spectral operators of matrices~\citeappx{DinSunSunToh20}. Let us further note that the stated selection property is always satisfied when $g$ is a piecewise linear mapping. In this case, the sets $\partial g(x)$ are polyhedral and outer Lipschitz continuity follows from Theorem~3D.1 of~\citeappx{DonRoc14}.

\section{Ablation Study on Parameter Choices}
\label{app:ablation}
This subsection provides more numerical experiments on the parameters of our method, using the examples given in Figs.~\ref{fig:LogisticRegression},  \ref{fig:TVDenoising} and \ref{fig:NNLS} of the paper. In each experiment, we run our method by varying a subset of the parameters while keeping all other parameters the same as in the original figures, to evaluate how the varied parameters influence the performance of our method. 
For the evaluation, we plot the same convergence graphs as in the original figures to compare the performance resulting from different parameter choices.
The evaluation is performed on the parameters $c$, $(p_1, p_2)$, $(\eta_1, \eta_2)$, and $\mu_0$.

\begin{figure}[t!]
	\centering
	\includegraphics[width=0.9\columnwidth]{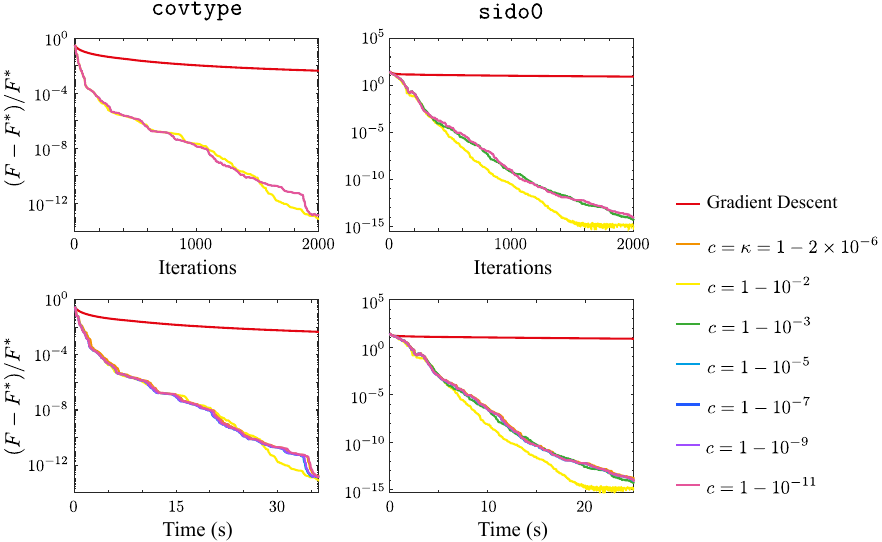}
	\caption{Convergence plots for LM-AA ($m = 15$) on the gradient descent solvers in Fig.~\ref{fig:LogisticRegression} with $\tau = L_F \times 10^{-6}$, using different values of the parameter $c$. }
	\label{fig:Ablation_LR_c}
\end{figure}

\begin{figure}[t!]
	\centering
	\includegraphics[width=0.9\columnwidth]{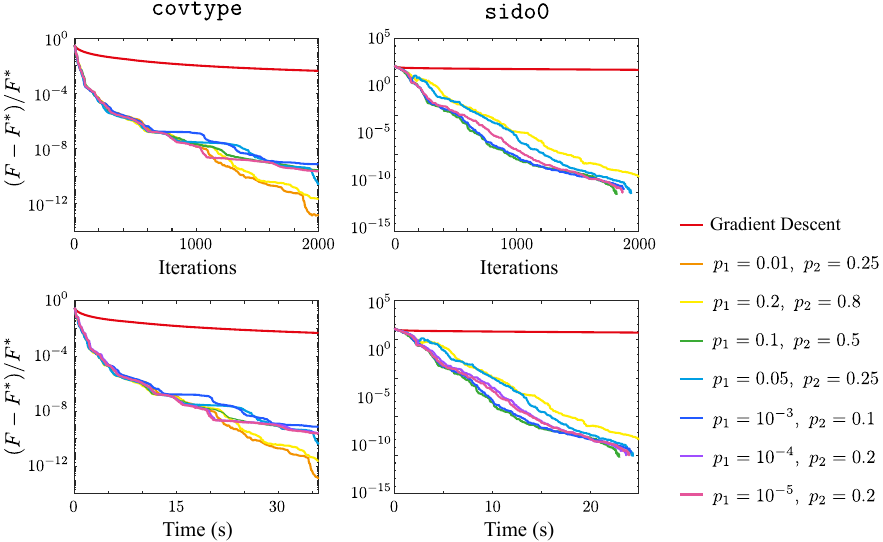}
	\caption{Convergence plots for LM-AA ($m = 15$) on the gradient descent solvers in Fig.~\ref{fig:LogisticRegression} with $\tau = L_F \times 10^{-6}$, using different values of the parameters $p_1$ and $p_2$. }
	\label{fig:Ablation_LR_p}
\end{figure}

We first consider the logistic regression problem in Fig.~\ref{fig:LogisticRegression} with $m=15$ and $\tau = L_F \times 10^{-6}$. The parameters used in Fig.~\ref{fig:LogisticRegression} are: $p_1=0.01$, $p_2=0.25$, $\eta_1=2$, $\eta_2=0.25$, $\mu_0 = 100$, $c = \kappa = ({L_F-\tau})/(L_F+\tau)$. Figs.~\ref{fig:Ablation_LR_c}, \ref{fig:Ablation_LR_p}, \ref{fig:Ablation_LR_eta}, and \ref{fig:Ablation_LR_mu} show the results using varied values of $c$, $(p_1, p_2)$, $(\eta_1, \eta_2)$, and $\mu_0$, respectively. 

\begin{figure}[t!]
	\centering
	\includegraphics[width=0.9\columnwidth]{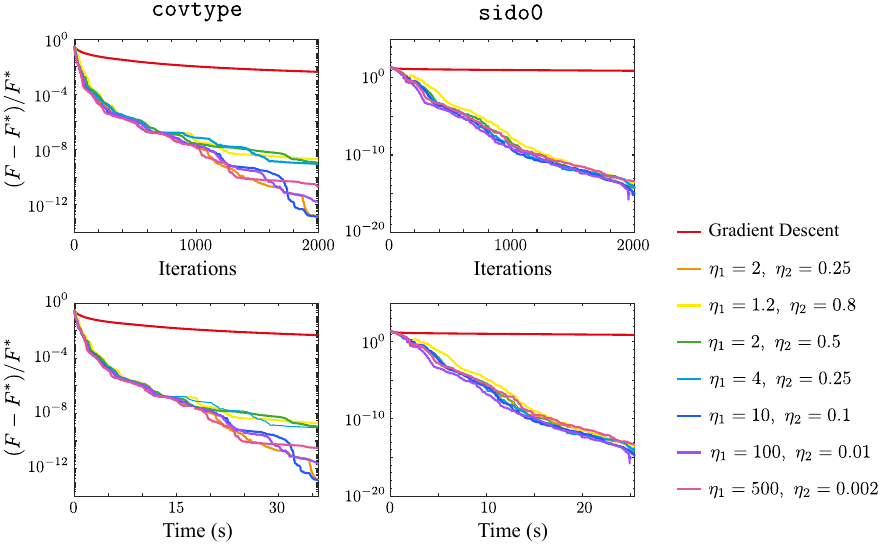}
	\caption{Convergence plots for LM-AA ($m = 15$) on the gradient descent solvers in Fig.~\ref{fig:LogisticRegression} with $\tau = L_F \times 10^{-6}$, using different values of the parameters $\eta_1$ and $\eta_2$. }
	\label{fig:Ablation_LR_eta}
\end{figure}

\begin{figure}[t!]
	\centering
	\includegraphics[width=0.85\columnwidth]{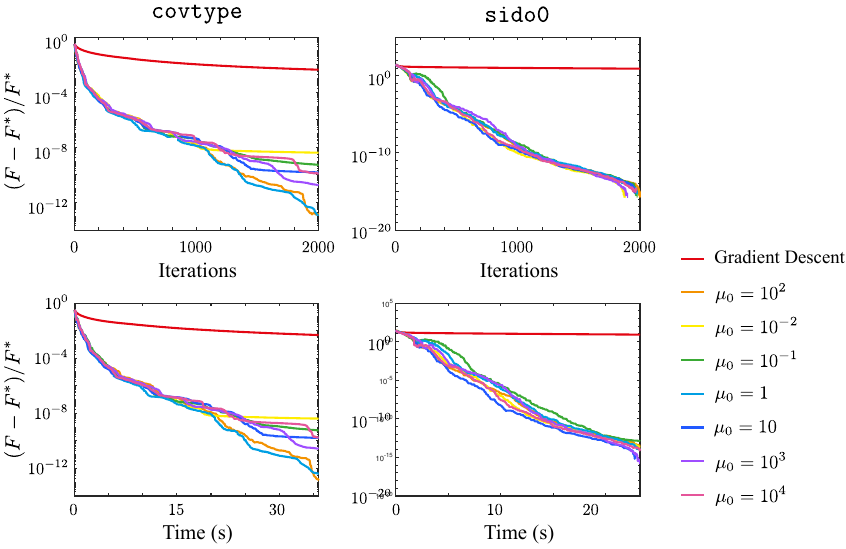}
	\caption{Convergence plots for LM-AA ($m = 15$) on the gradient descent solvers in Fig.~\ref{fig:LogisticRegression} with $\tau = L_F \times 10^{-6}$, using different values of the parameter $\mu_0$.}
	\label{fig:Ablation_LR_mu}
\end{figure}

Next, we consider the image reconstruction problem in Fig.~\ref{fig:TVDenoising} with $m = 5$ and $\beta = 100$. The parameters used in Fig.~\ref{fig:TVDenoising} are: $p_1=0.01$, $p_2=0.25$, $\eta_1=2$, $\eta_2=0.25$, $\mu_0 = 1$, and $c = \kappa$ where $\kappa$ is derived in Appendix~\ref{sec:app-tv}. Figs.~\ref{fig:Ablation_TV_c}, \ref{fig:Ablation_TV_p}, \ref{fig:Ablation_TV_eta}, and \ref{fig:Ablation_TV_mu} show the results using varied values of $c$, $(p_1, p_2)$, $(\eta_1, \eta_2)$, and $\mu_0$, respectively.

\begin{figure}[t!]
	\centering
	\includegraphics[width=0.9\columnwidth]{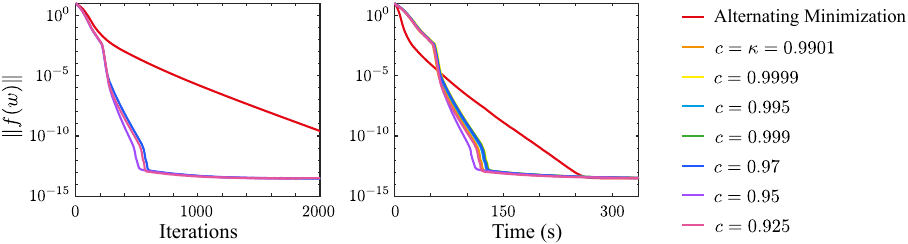}
	\caption{Convergence plots for LM-AA ($m = 5$) on the alternating minimization solver in Fig.~\ref{fig:TVDenoising} with $\beta = 100$, using different values of the parameter $c$.}
	\label{fig:Ablation_TV_c}
\end{figure}

\begin{figure}[t!]
	\centering
	\includegraphics[width=0.9\columnwidth]{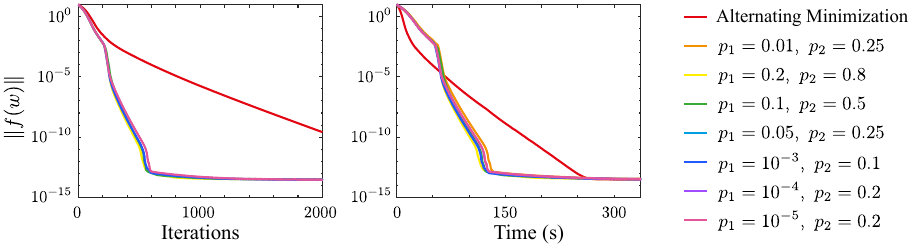}
	\caption{Convergence plots for LM-AA ($m = 5$) on the alternating minimization solver in Fig.~\ref{fig:TVDenoising} with $\beta = 100$, using different values of the parameters $p_1$ and $p_2$.}
	\label{fig:Ablation_TV_p}
\end{figure}

\begin{figure}[t!]
	\centering
	\includegraphics[width=0.9\columnwidth]{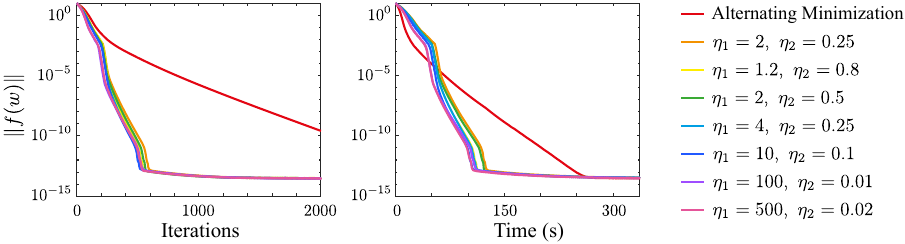}
	\caption{Convergence plots for LM-AA ($m = 5$) on the alternating minimization solver in Fig.~\ref{fig:TVDenoising} with $\beta = 100$, using different values of the parameters $\eta_1$ and $\eta_2$.}
	\label{fig:Ablation_TV_eta}
\end{figure}

\begin{figure}[t!]
	\centering
	\includegraphics[width=0.9\columnwidth]{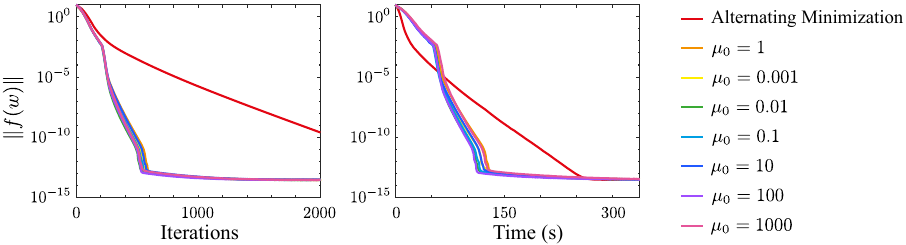}
	\caption{Convergence plots for LM-AA ($m = 5$) on the alternating minimization solver in Fig.~\ref{fig:TVDenoising} with $\beta = 100$, using different values of the parameter $\mu_0$.}
	\label{fig:Ablation_TV_mu}
\end{figure}

Finally, we consider the nonnegative least squares problem in Fig.~\ref{fig:NNLS} with $m=10$. The parameters used in Fig.~\ref{fig:NNLS} are: $p_1=0.01$, $p_2=0.25$, $\eta_1=2$, $\eta_2=0.25$, $\mu_0 = 1$, and $c = \kappa$ where $\kappa$ is derived in Appendix~\ref{appx:NNLS}. Figs.~\ref{fig:Ablation_NNLS_c}, \ref{fig:Ablation_NNLS_p}, \ref{fig:Ablation_NNLS_eta}, and \ref{fig:Ablation_NNLS_mu} show the results using varied values of $c$, $(p_1, p_2)$, $(\eta_1, \eta_2)$, and $\mu_0$, respectively.

\begin{figure}[t!]
	\centering
	\includegraphics[width=0.93\columnwidth]{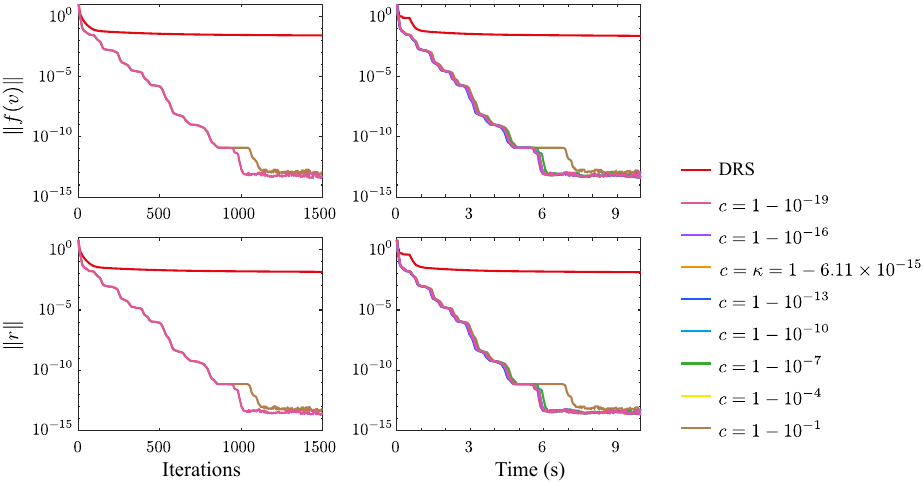}
	\caption{Convergence plots for LM-AA ($m = 10$) on the DR splitting solver in Fig.~\ref{fig:NNLS}, using different values of the parameter $c$.}
	\label{fig:Ablation_NNLS_c}
\end{figure}

\begin{figure}[t!]
	\centering
	\includegraphics[width=0.9\columnwidth]{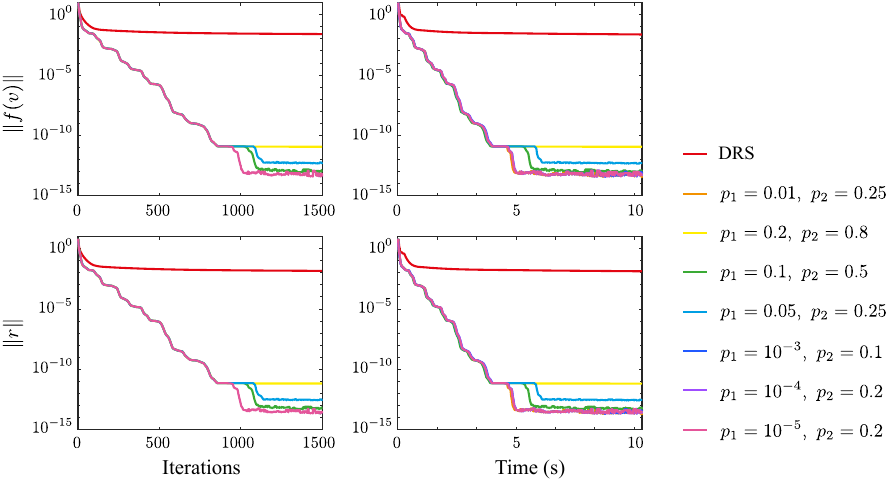}
	\caption{Convergence plots for LM-AA ($m = 10$) on the DR splitting solver in Fig.~\ref{fig:NNLS}, using different values of the parameters $p_1$ and $p_2$.}
	\label{fig:Ablation_NNLS_p}
\end{figure}

\begin{figure}[t!]
	\centering
	\includegraphics[width=0.9\columnwidth]{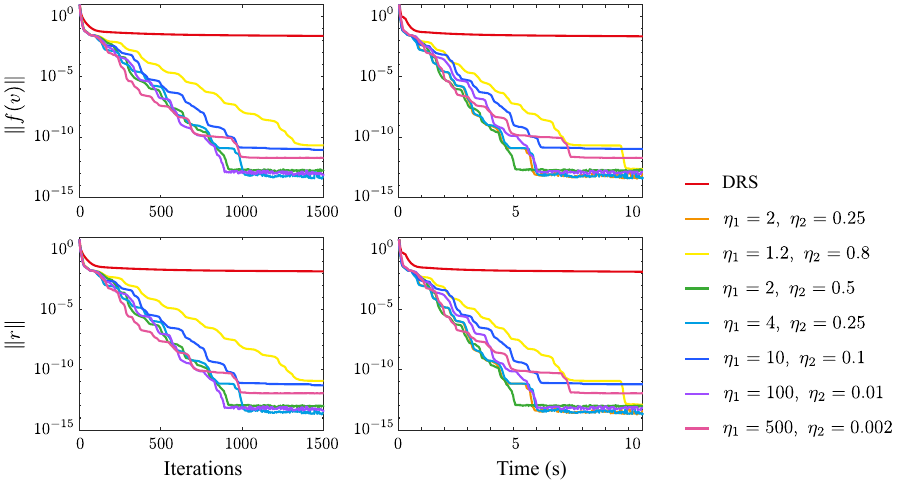}
	\caption{Convergence plots for LM-AA ($m = 10$) on the DR splitting solver in Fig.~\ref{fig:NNLS}, using different values of the parameters $\eta_1$ and $\eta_2$.}
	\label{fig:Ablation_NNLS_eta}
\end{figure}

\begin{figure}[t!]
	\centering
	\includegraphics[width=0.81\columnwidth]{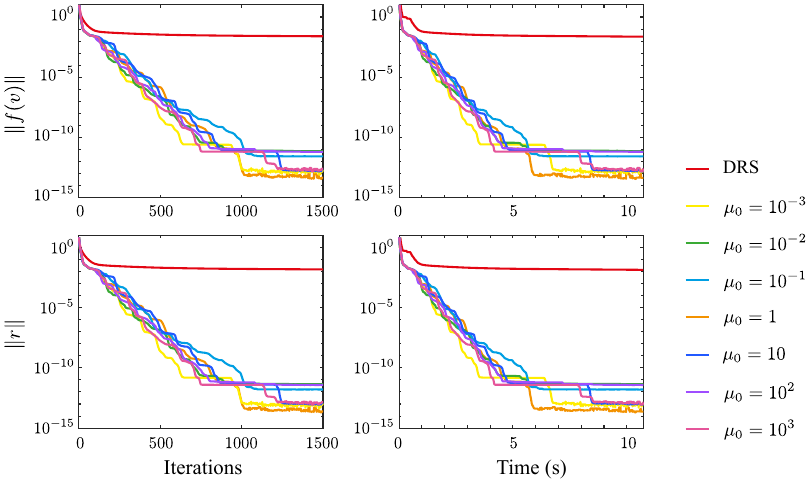}
	\caption{Convergence plots for LM-AA ($m = 10$) on the DR splitting solver in Fig.~\ref{fig:NNLS}, using different values of the parameter $\mu_0$.}
	\label{fig:Ablation_NNLS_mu}
\end{figure}

As shown in Figs.~\ref{fig:Ablation_LR_c}, \ref{fig:Ablation_TV_c}, and \ref{fig:Ablation_NNLS_c}, our algorithm is not very sensitive w.r.t. the choice of $c$. Specifically, we still observe convergence if $c$ is chosen smaller than the Lipschitz constant $\kappa$. This robustness is particularly important when good estimates of the constant $\kappa$ are not available. Although the bound $c \geq \kappa$ is required in our theoretical results in Section~\ref{sec:loc}, fast convergence can still be observed for other choices of $c$. This indicates that either we have over-estimated the Lipschitz constant $\kappa$ or the acceleration effect of $\AA$ steps can be much faster than $\kappa$. 

The results in Figs.~\ref{fig:Ablation_LR_p}, \ref{fig:Ablation_TV_p}, and \ref{fig:Ablation_NNLS_p} demonstrate that the performance of LM-AA is not overly affected by the choice of the trust-region parameters $p_1$ and $p_2$ either. In general, good performance can be achieved if $p_1$ is moderately small and $p_2$ is not too large. Thus, we decide to work with the standard choice $p_1 = 0.01$ and $p_2 = 0.25$. 

In comparison, the trust-region parameters $\eta_1$ and $\eta_2$ can have a more significant impact on the performance of our algorithm. While the performance of LM-AA is not sensitive to the choice of $\eta_1$ and $\eta_2$ in the logistic regression problem using the dataset \texttt{sido0} (Fig.~\ref{fig:Ablation_LR_eta}) and in the denoising problem (Fig.~\ref{fig:Ablation_TV_eta}), more variation can be seen in the remaining two examples. In general, the performance seems to deteriorate when $\eta_1$ and $\eta_2$ are chosen to be very close to each other. The standard choice $\eta_1 = 2$ and $\eta_2 = 0.25$ again achieves convincing performance on all numerical examples and has a good balance when increasing and decreasing the weight parameter $\lambda_k$. 

The convergence plot for different values of $\mu_0$ are shown in Figs.~\ref{fig:Ablation_LR_mu}, \ref{fig:Ablation_TV_mu}, and \ref{fig:Ablation_NNLS_mu}. Our observations are again somewhat similar: the performance of LM-AA on the logistic regression problem for \texttt{sido0} (Fig.~\ref{fig:Ablation_LR_mu}) and on the denoising problem  (Fig.~\ref{fig:Ablation_TV_mu}) is very robust w.r.t. the choice of $\mu_0$. In the nonnegative least squares problem, $\mu_0$ appears to mainly affect the last convergence stage of the algorithm, i.e., different choices of $\mu_0$ can lead to an earlier jump to a level with higher accuracy. Overall the parameters $\mu_0 =1$ (for denoising and NNLS) and $\mu_0 = 100$ (for logistic regression) yield the most robust results.

\section{Ablation Study on Permutation Strategy}
\label{appx:AblationPermutation}
{We also provide an ablation study for the permutation strategy in line 6 of Algorithm~\ref{algo1}. In general, when the parameter $\gamma$ is small, our nonmonotone globalization strategy is close to a monotone criterion on the residual. In this case, the minimal residual iteration $k_0$ mostly coincides with the current iteration number $k$ and the permutation causes little difference. However, when $\gamma$ is (relatively) large, then the usage of permutations can cause essential differences in the numerical performance. In particular, when the current trial step is rejected, then Algorithm~\ref{algo1} performs $x^{k+1} = g^{k}$ as next iteration if no permutation is used. In general, the update $x^{k+1} = g^{k}$ can be worse than $x^{k+1} =g^{k_0}$ since $x^{k_0}$ has the smallest residual among the $m$ latest iterations. We test Algorithm~\ref{algo1} without permutation in Figure~\ref{fig:permutation} for the logistic regression experiment. We set $\gamma=0.05$ in Figure~\ref{fig:permutation} for LM-AA and keep other parameters unchanged. As can be seen from the figure, permutation improves the overall convergence and performance.} 

\begin{figure} [t!]
	\centering
	\includegraphics[width=\columnwidth]{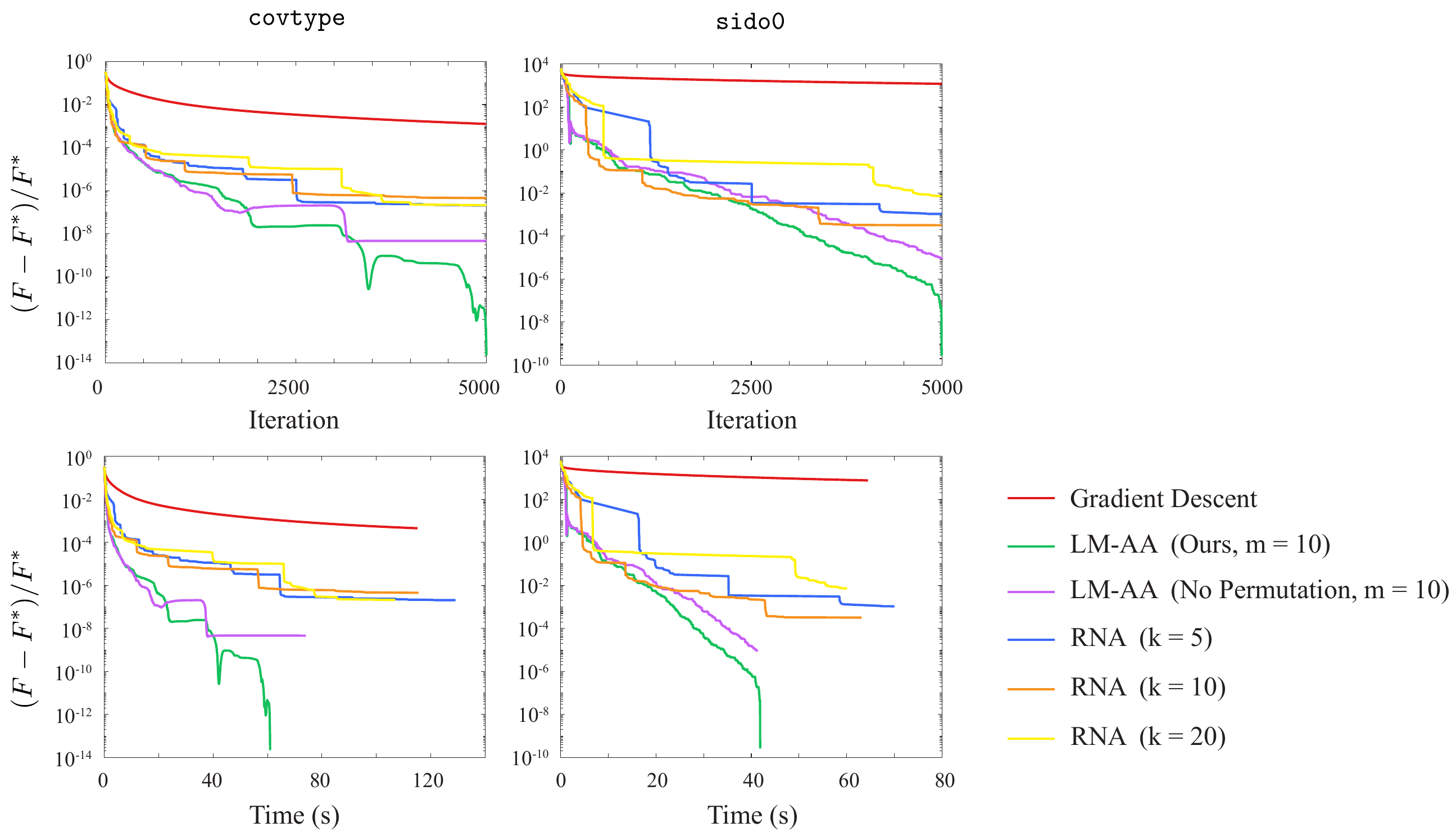}
	\caption{{The performance of LM-AA ($m=10, \gamma=0.05$) with or without the permutation strategy given in Eq.~\ref{eq:PermutationRule}, for logistic regression on the \texttt{covtype} and \texttt{sido0} datasets.}}
	\label{fig:permutation} 
\end{figure}

\end{document}